\documentclass[letter,11pt]{amsart}

\usepackage{amsaddr}

\usepackage{fullpage}			
\usepackage{graphicx}			
\usepackage[english]{babel}
\usepackage[T1]{fontenc}		
\usepackage[utf8]{inputenc}		
\usepackage{amsmath}			
\usepackage{amssymb}			
\usepackage{amsthm}			
\usepackage{mathrsfs}			
\usepackage[bookmarksnumbered]{hyperref}		
\usepackage{eucal}			
\usepackage{enumitem}
\usepackage[normalem]{ulem}

\usepackage[numbers]{natbib}

\newcounter{generalnumbering} \numberwithin{generalnumbering}{section}

\theoremstyle{plain}		\newtheorem{theorem}[generalnumbering]{Theorem}
\theoremstyle{plain}		\newtheorem{corollary}[generalnumbering]{Corollary}
\theoremstyle{definition}		\newtheorem{definition}[generalnumbering]{Definition}
\theoremstyle{definition}		\newtheorem{example}[generalnumbering]{Example}
\theoremstyle{plain}		\newtheorem{proposition}[generalnumbering]{Proposition}
\theoremstyle{plain}		\newtheorem{lemma}[generalnumbering]{Lemma}

\newcommand{\namedenvironment}{}
\theoremstyle{definition}
\newtheorem*{theoremfornamed}{\namedenvironment}
\newenvironment{named}[1]{\renewcommand{\namedenvironment}{#1}\begin{theoremfornamed}}{\end{theoremfornamed}}

\newcounter{counterequation}[section]
\newcommand{\ntag}{\stepcounter{counterequation}\tag{\arabic{section}.\arabic{counterequation}}}

\makeatletter
\def\G{\@ifnextchar[{\@Gwithbrak}{\@Gwithoutbrak}}
\def\@Gwithbrak[#1]{\mathcal{G}^{(#1)}}
\def\@Gwithoutbrak{\mathcal{G}}
\makeatother

\makeatletter
\def\H{\@ifnextchar[{\@Hwithbrak}{\@Hwithoutbrak}}
\def\@Hwithbrak[#1]{\mathcal{H}^{(#1)}}
\def\@Hwithoutbrak{\mathcal{H}}
\makeatother

\newcommand{\perpp}{\perp\!\!\!\perp}
\DeclareMathOperator{\supp}{supp}
\DeclareMathOperator{\so}{\mathfrak{s}}
\DeclareMathOperator{\ra}{\mathfrak{r}}
\DeclareMathOperator{\KB}{\bf KB}
\newcommand{\id}{\mathrm{id}}

\title{A general Banach--Stone type theorem and applications}
\author{Luiz Gustavo Cordeiro}
\thanks{Supported by CAPES/Ciência Sem Fronteiras PhD scholarship 012035/2013-00, and by ANR project GAMME (ANR-14-CE25-0004)}
\address{UMPA, UMR 5669 CNRS -- École Normale Supérieure de Lyon\\
	46 allée d'Italie, 69364 Lyon Cedex 07, France}

\email{luizgc6@gmail.com}				


\begin{document}

\begin{abstract}
	One important class of tools in the study of the connections between algebraic and topological structures are the ``Banach--Stone type theorems'', which describe algebraic isomorphisms of algebras (or groups, lattices, etc.) of functions in terms of homeomorphisms between the underlying topological spaces. Several such theorems have been proven throughout the last century, however not all of them are comparable, and in particular no single one is the strongest. In this article, we describe a general framework which encompasses several of these results, and which allows for new applications related to groupoid algebras, and to groups of circle-valued functions. This is attained by a detailed study of ``disjointness relations'' on sets of functions, which play a central role (even if not explicitly) in previously-proven Banach--Stone type theorems.
\end{abstract}

\maketitle

\subsubsection*{Keywords:} Locally compact Hausdorff; Banach--Stone; Function space; Recovery theorem; Groupoids.

\section*{Introduction}

Let $X$ be a locally compact and Hausdorff space, and consider $\mathbb{K}=\mathbb{R}$ or $\mathbb{C}$. The initial motivation for this work is the question of whether we can recover $X$ (up to homeomorphism) from $C_c(X,\mathbb{K})$, the set of continuous, compactly supported $\mathbb{K}$-valued functions on $X$. This is a problem initiated mostly after Stone's seminal work \cite{MR1501865} on the representations of Boolean algebras, and has proven to be a rich area of study with several important applications in Logic, Functional Analysis, and Operator Algebras.

By Milutin's Theorem (\cite{MR0206695}, or \cite[Chapter 36, Theorem 2.1]{MR1999613}), just the topological vector space structure of $C_c(X,\mathbb{K})$, when endowed with the supremum norm, is not enough to recover $X$. On the other hand, throughout the last century several authors have proved that by considering additional algebraic structures on $C_c(X,\mathbb{K})$ -- such as that of a ring, a C*-algebra if $\mathbb{K}=\mathbb{C}$, a lattice if $\mathbb{K}=\mathbb{R}$, etc. -- we can in fact recover $X$. See Banach and Stone \cite{MR1357166,MR1501905}, Gelfand and Kolmogorov \cite{gelfandkolmogorov1939}, Milgram \cite{MR0029476}, Gelfand and Naimark \cite{MR0009426}, Kaplansky \cite{MR0020715}, Jarosz \cite{MR1060366}, Li and Wong \cite{MR3162258}, Hernández and Ródenas \cite{MR2324919}, Kania and Rmoutil \cite{MR3813611}.

In fact, the results of \cite{MR0020715,MR2324919} also hold for certain spaces of non-$\mathbb{R}$ or $\mathbb{C}$-valued functions. In a similar manner, Stone's duality for Boolean algebras (\cite{MR1501865}) can also be seen as a result on spaces of functions: The Boolean algebra of clopen sets of a topological space $X$ is order-isomorphic to the lattice of functions $C(X,\left\{0,1\right\})$, and if $X$ is a \emph{Stone} (zero-dimensional, compact Hausdorff) space, then it completely determines $X$.

Our goal in this paper is to provide a unified and elementary approach to all these results, under hypotheses that can be easily verified in different settings. For this, we use a stronger version of the ``disjointness'' relation for (supports of) functions as considered by Jarosz in \cite{MR1060366}. As we will see in Section \ref{sectionconsequences}, all of the results above immediately fall in this more general setting.

Let us describe the main idea in the case real-valued functions: Two functions $f,g\in C_c(X,\mathbb{R})$ are \emph{strongly disjoint} if $\supp(f)\cap\supp(g)=\varnothing$, in which case we write $f\perpp g$. Then it is possible to describe, purely in terms of the relation $\perpp$, the subsets of  $C_c(X,\mathbb{R})$ of the form $\mathbf{I}(U)=\left\{f:\supp(f)\subseteq U\right\}$, where $U\subseteq X$ is open. These sets are called \emph{$\perpp$-ideals}, and we have a bijection $x\mapsto \mathbf{I}(X\setminus\left\{x\right\})$ between $X$ and the set $\widehat{C_c(X,\mathbb{R})}$ of maximal $\perpp$-ideals. This bijection can be made into a homeomorphism, by endowing $\widehat{C_c(X,\mathbb{R})}$ with a topology, described again only in terms of $\perpp$. Therefore, any $\perpp$-isomorphism $T\colon C_c(X,\mathbb{R})\to C_c(Y,\mathbb{R})$ will induce a homeomorphism $Y\cong\widehat{C_c(Y,\mathbb{R})}\cong \widehat{C_c(X,\mathbb{R})}\cong X$. In all of the previously-proven theorems listed above, the algebraic isomorphisms under considerations happen to be also $\perpp$-isomorphisms, and thus those results follow from this.

As supports of functions are central to the result above, and we wish also to look at a theory involving non-scalar maps, we will need to extend the notion of support, which is the first problem tackled in Section \ref{sectiondisjointness}.

This article is organized as follows: In the \hyperref[sectiondisjointness]{first section} we introduce all necessary terminology and prove our main recovery theorem (Theorem \ref{maintheorem}). In Section \ref{sectionbasicmaps}, we study an important class of maps, called ``basic'', between spaces of functions, and which will appear in most applications. These are the maps which are ``classifiable'' in a certain manner.

Due to the level of generality we seek, the first two sections are rather abstract, so the reader is invited to read Definition \ref{definitionsigma} in order to get familiarized with the notation, read the main Theorem \ref{maintheorem}, and proceed directly to the applications in Section \ref{sectionconsequences}, referring back to previous parts of this article as necessary (or desired).

In Section \ref{sectionconsequences} we obtain classifications of isomorphisms for different algebraic structures on spaces of continuous functions, including the ones mentioned at the beginning of this introduction. The new applications consist of a classification of linear isomorphisms which are isometric with respect to $L^1$-norms (Theorem \ref{theoremdisjointnessl1}), classifications of classes of isomorphisms of algebras associated to groupoids (Theorems \ref{theoremmeasuredgroupoidconvolutionalgebra}, \ref{theoremrenault} and \ref{theoremsteinbergalgebras}), and a classification of (uniform-metric) isometric isomorphisms between groups of circle-valued functions.

\section{Disjointness and \texorpdfstring{$\perpp$}{⊥⊥}-isomorphisms}\label{sectiondisjointness}

Throughout this section, $X$ will always denote a locally compact and Hausdorff space, $H$ will denote a Hausdorff space, and $\theta\colon X\to H$ is a fixed continuous function. We denote by $C(X,H)$ the set of continuous functions from $X$ to $H$. For two functions $f,g\colon X\to H$, we denote
\[[f\neq g]=\left\{x\in X:f(x)\neq g(x)\right\}\quad\text{and}\quad[f=g]=X\setminus[f\neq g].\]

\subsection{Disjointness relations}

The main idea is that any manageable algebraic structure on $H$ will naturally lead us to consider a specific function $\theta$, which behaves as a ``neutral element'', and which will be used to separate points of $X$. See Example \ref{examplethetaiszero} for the classical setting. We first generalize the notion of support, in the obvious manner.

\begin{definition}\label{definitionsigma}
	Given $f\in C(X,H)$, we define the \emph{$\theta$-support} of $f$ as
	\[\supp^\theta(f)=\overline{[f\neq\theta]}.\]
	We define $\sigma^\theta(f)$ as the interior of $\supp^\theta(f)$, and $Z^\theta(f)$ as the complement of $\supp^\theta(f)$:
	\[\sigma^\theta(f)=\operatorname{int}\supp^\theta(f)\qquad\text{and}\qquad Z^\theta(f)=X\setminus\supp^\theta(f).\]
	
	Let $C_c(X,\theta)$ be the set of continuous functions from $X$ to $H$ with compact $\theta$-support. Whenever there is no risk of confusion, we will drop $\theta$ from the notation and write simply $\supp(f)$, $\sigma(f)$, $Z(f)$ and $C_c(X)$.
	
	Now we define the following relations: Given $f,g\in C(X,H)$,
	\begin{enumerate}[label=(\arabic*)]
		\item $f\perp g$: if $[f\neq\theta]\cap[g\neq\theta]=\varnothing$; we say that $f$ and $g$ are \emph{weakly disjoint};\index{Disjoint functions}
		\item $f\perpp g$: if $\supp(f)\cap\supp(g)=\varnothing$; we say that $f$ and $g$ are \emph{strongly disjoint};\index{Disjoint functions!Strongly disjoint}
		\item $f\subseteq g$: if $\sigma(f)\subseteq\sigma(g)$;
		\item $f\Subset g$: if $\supp(f)\subseteq\sigma(g)$.
	\end{enumerate}
\end{definition}

Note that $Z^\theta(f)$ is the complement of $\sigma^\theta(f)$ in the lattice of regular open sets of $X$ (see \cite[Chapter 10]{MR2466574}). Also, $\sigma^\theta(f)$ is the regularization of $[f\neq\theta]$, and thus it follows immediately that
\[f\perp g\iff \sigma(f)\cap\sigma(g)=\varnothing,\]
even though $[f\neq \theta]$ and $\sigma^\theta(f)$ are not equal in general.

\begin{example}
	Suppose $X=H=[0,1]$, $\theta=0$ (the zero map $[0,1]\to[0,1]$) and $f=\id_{[0,1]}$, the identity map of $[0,1]$. Then $[f\neq\theta]=(0,1]$ but $\sigma^\theta(f)=[0,1]$.
\end{example}

As stated above, when $H$ comes with additional structure, a particular choice of $\theta$ generally yields a suitable notion of support, and the relations above may be described in terms of this structure. This is the general technique used in the  applications in Section \ref{sectionconsequences}.

\begin{example}\label{examplethetaiszero}
	If $H=\mathbb{R}$ or $\mathbb{C}$, and $\theta=0$ is the constant zero function, we obtain the usual notion of support. We may describe $\perp$ in terms of the multiplicative structure of $C_c(X)=C_c(X,0)$: $f\perp g$ if and only if $fg=0$, which is the only absorbing element of $C_c(X)$.
\end{example}

\begin{example}[Kania--Rmoutil, \cite{MR3813611}]\label{examplekaniarmoutil}
	Let $X$, $H$ and $\theta$ as in the beginning of this \hyperref[sectiondisjointness]{section}. Define the \emph{compatibility ordering} on $C_c(X,\theta)$ by
	\[f\preceq g\iff g|_{\supp^\theta(f)}=f|_{\supp^\theta(f)}.\]
	Then $\theta$ is the minimum of $\preceq$ in $C_c(X,\theta)$. We can describe weak disjointness in $C_c(X,\theta)$ by
	\[f\perp g\iff\inf_{\preceq}\left\{f,g\right\}=\theta\text{ and }\left\{f,g\right\}\text{ has a }\preceq\text{-upper bound.}\]
\end{example}

We will, moreover, be interested in recovering $X$ not from the whole set $C_c(X)$, but instead from a subcollection $\mathcal{A}\subseteq C_c(X)$. We will need to assume, however, that there are enough functions in $\mathcal{A}$ in order to separate points of $X$, and this is attained by assuming that an appropriate version of Urysohn's Lemma is valid. (This is the same type of assumption as made in \cite{MR0020715} and in \cite{MR2324919}.)

\begin{definition}\label{definitionregularperpp}
	Let $\mathcal{A}\subseteq C_c(X)$ be a subset containing $\theta$. Denote $\sigma(\mathcal{A})=\left\{\sigma(f):f\in\mathcal{A}\right\}$. We say that $(X,\theta,\mathcal{A})$ (or simply $\mathcal{A}$) is
	\begin{enumerate}
		\item \emph{weakly regular} if $\sigma(\mathcal{A})$ is a basis for the topology of $X$.\index{Regular family of functions!weakly regular}
		\item \emph{regular} if for every $x\in X$, every neighbourhood $U$ of $x$ and every $c\in H$ there is $f\in\mathcal{A}$ with $f(x)=c$ and $\supp(f)\subseteq U$.\index{Regular family of functions}
	\end{enumerate}
\end{definition}

We will need to analyze relations between $\subseteq,\Subset,\perp$ and $\perpp$. In order to deal with finitely many functions simultaneously, we will need to adapt the notion of \emph{open cover} of a set to this language.

\begin{definition}\label{definitiondisjointcover}\index{Cover}
	Suppose $\mathcal{A}\subseteq C_c(X)$ is weakly regular. A family $A\subseteq\mathcal{A}$ is a \emph{cover} of an element $b\in\mathcal{A}$ if the implication
	\[h\perp a\text{ for all }a\in A\Longrightarrow h\perp b.\]
	is valid for all $h\in\mathcal{A}$.
\end{definition}

\begin{lemma}\label{lemmacovers}
	Suppose $\mathcal{A}$ is weakly regular, and let $A\subseteq\mathcal{A}$ and $b\in\mathcal{A}$. The following are equivalent:
	\begin{enumerate}[label=(\arabic*)]
		\item\label{lemmacovers(1)} $A$ is a cover of $b$;
		\item\label{lemmacovers(2)} The closure of $\bigcup_{a\in A}[a\neq\theta]$ contains $\supp(b)$.
	\end{enumerate}
\end{lemma}
\begin{proof}
	\ref{lemmacovers(1)}$\Rightarrow$\ref{lemmacovers(2)}: Let $x\in\supp(b)$. Take an open neighbourhood of $x$ of the form $\sigma(h)$, $h\in\mathcal{A}$. Since $\supp(b)=\overline{\sigma(b)}$, the intersection $\sigma(h)\cap\sigma(b)$ is nonempty and thus $h$ and $b$ are not weakly disjoint. From $A$ being a cover, $h$ is not weakly disjoint to some $a\in A$, which means that $\sigma(h)\cap[a\neq\theta]$ is nonempty. Since $\mathcal{A}$ is weakly regular, then $x$ is in the closure of $\bigcup_{a\in A}[a\neq\theta]$.
	
	\ref{lemmacovers(2)}$\Rightarrow$\ref{lemmacovers(1)}: Suppose $h\in\mathcal{A}$ is such that $h\perp a$ for all $a\in A$. This means that
	\[(\bigcup_{a\in A}[a\neq\theta])\cap[h\neq\theta]=\varnothing.\]
	Taking the closure of the first term and using \ref{lemmacovers(2)} we conclude that $[b\neq\theta]\cap[h\neq\theta]\subseteq\supp(b)\cap[h\neq\theta]=\varnothing$, so $h\perp b$.
\end{proof}

If $\mathcal{A}\subseteq C_c(X)$ and $\theta\in\mathcal{A}$, note that $\subseteq$ is a preorder on $\mathcal{A}$, whose minimum is $\theta$. Alternatively, $\theta$ is the only element of $\mathcal{A}$ such that $\theta\perp\theta$. Thus the function $\theta$ is uniquely determined in terms of either $\perp$ or $\subseteq$. We now proceed to prove that $\subseteq$ and $\perp$ ``carry the same information'', as do $\Subset$ and $\perpp$.

\begin{proposition}\label{propositionrelationsrelations}
	Suppose $\mathcal{A}$ is weakly regular. If $f,g\in\mathcal{A}$, then
	\begin{enumerate}[label=(\alph*)]
		\item\label{theoremrelationsrelations(a)} $f\subseteq g\iff\forall h(h\perp g\Rightarrow h\perp f)$;
		\item\label{theoremrelationsrelations(b)} $f\perp g\iff$ The $\subseteq$-infimum of $\left\{f,g\right\}$ is $\theta$;
		\item\label{theoremrelationsrelations(c)} $f\subseteq g\iff\forall h(h\Subset f\Rightarrow h\Subset g)$;
		\item\label{theoremrelationsrelations(d)} $f\subseteq g\iff\forall h(h\perpp g\Rightarrow h\perpp f)$;
		\item\label{theoremrelationsrelations(e)} $f\perpp g\iff\exists h_1,k_1,\ldots,h_n,k_n\in\mathcal{A}$ such that $\left\{h_1,\ldots,h_n\right\}$ is a cover of $f$, $h_i\Subset k_i$ and $\phantom{f\perpp g\iff...}k_i\perp g$ for all $i$;
		\item\label{theoremrelationsrelations(f)} $f\Subset g\iff\forall b\in\mathcal{A}$, $\exists h_1,\ldots,h_n\in\mathcal{A}$ such that $\left\{h_1,\ldots,h_n,g\right\}$ is a cover of $b$ and $h_i\perpp f$.
	\end{enumerate}
	By items \ref{theoremrelationsrelations(a)} and \ref{theoremrelationsrelations(b)}, $\perp$ and $\subseteq$ are equi-expressible (i.e., each one is completely determined by the other). By \ref{theoremrelationsrelations(c)} and \ref{theoremrelationsrelations(d)} one can recover $\subseteq$ (and hence $\perp$) from either $\Subset$ or $\perpp$, which in turn implies, from \ref{theoremrelationsrelations(e)} and \ref{theoremrelationsrelations(f)}, that $\Subset$ and $\perpp$ are also equi-expressible.
\end{proposition}
\begin{proof}
	Items \ref{theoremrelationsrelations(a)}-\ref{theoremrelationsrelations(d)} are easy consequences of weak regularity of $\mathcal{A}$, and $X$ being a regular topological space for items \ref{theoremrelationsrelations(c)}-\ref{theoremrelationsrelations(d)}.
	\begin{enumerate}[label=\ref{theoremrelationsrelations(\alph*)}]\setcounter{enumi}{4}
		\item $\Rightarrow$: Suppose $f\perpp g$. Given $x\in \supp(f)$, weak regularity of $\mathcal{A}$ and regularity of the topological space $X$ give us $h_x,k_x\in\mathcal{A}$ such that $x\in\sigma(h_x)$, $h_x\Subset k_x$ and $k_x\perp g$. Compactness of $\supp(f)$ allows us to find the elements $h_i,k_i$ we need, by going to a subcover of $\left\{\sigma(h_x):x\in\supp(f)\right\}$.
		
		$\Leftarrow$: Suppose such $h_i,k_i$ exist. Then by Lemma \ref{lemmacovers},
		\[\supp(f)\subseteq\bigcup_{i=1}^n\supp(h_i)\subseteq\bigcup_{i=1}^n\sigma(k_i)\subseteq X\setminus\supp(g),\]
		and so $f\perpp g$.
		\item $\Rightarrow$: Suppose $f\Subset g$ and take any $b\in\mathcal{A}$. Since $\supp(b)\setminus\sigma(g)$ is compact and does not intersect $\supp(f)$, we can take $h_1,\ldots,h_n\in\mathcal{A}$ such that $h_i\perpp f$ and $\supp(b)\setminus\sigma(g)\subseteq\bigcup_i\sigma(h_i)$, which implies that $\left\{h_1,\ldots,h_n,g\right\}$ is a cover of $b$.
		
		$\Leftarrow$: By compactness os $\supp(f)$ and $\supp(g)$, take $b_1,\ldots,b_M$ in $\mathcal{A}$ such that $\supp(f)\cup\supp(g)\subseteq\bigcup_{k=1}^M\sigma(b_k)$. For each $k$ take functions $h_i^k$ satisfying the right-hand side of \ref{theoremrelationsrelations(f)}, relative to $b_k$.
		
		Given $k$, we have $\sigma(b_k)\subseteq\bigcup_i\supp(h_i^k)\cup\supp(g)$, so by taking complements we obtain $\bigcap_iZ(h_i^k)\cap Z(g)\cap\sigma(b_k)=\varnothing$, or equivalently $\bigcap_iZ(h_i^k)\cap\sigma(b_k)\subseteq\sigma(b_k)\setminus Z(g)\subseteq\supp(g)$. Taking interiors on both sides yields $\bigcap_iZ(h_i^k)\cap\sigma(b_k)\subseteq\sigma(g)$.
		
		Now from $h_i^j\perpp f$ we obtain
		\begin{align*}
		\supp(f)&\subseteq\bigcap_{i,j}Z(h_i^j)\cap\bigcup_{k=1}^M\sigma(b_k)\subseteq\bigcup_{k=1}^M\left[\bigcap_i Z(h_i^k)\cap\sigma(b_k)\right]\subseteq \sigma(g),
		\end{align*}
		so $f\Subset g$.\qedhere
	\end{enumerate}
\end{proof}

\begin{named}{Remark}
	One should be careful with the connections between the pairs of relations $(\perp,\perpp)$ and $(\subseteq,\Subset)$. For example, $\perp$ and $\perpp$ may coincide but $\subseteq$ and $\Subset$ may not and vice-versa. See the example below.
\end{named}

\begin{example}
	Let $X=H=\mathbb{R}$ and $\theta=0$, so that we are dealing with the usual notion of support. Let $\left\{(a_n,b_n):n\in\mathbb{N}\right\}$ (where $a_n<b_n$) be a countable basis of open intervals for the usual topology of $\mathbb{R}$. Up to small modifications, we may assume that all the numbers $a_n$, $b_n$ and $b_n+1$ are distinct. In particular, the sets $U_n:=(a_n,b_n)$ have pairwise disjoint boundaries.
	
	For each $n$, let $f_n\in C_c(\mathbb{R})$ with $\sigma(f_n)=U_n$, e.g. $f_n(x)=\max(0,(x-\widetilde{a_n})(\widetilde{b_n}-x))$,
	and let $\mathcal{A}=\left\{f_n:n\in\mathbb{N}\right\}$, which is weakly regular. Then $\perp$ and $\perpp$ coincide on $\mathcal{A}$, as do $\subseteq$ and $\Subset$, since the boundaries of all $U_n$ are pairwise disjoint.
	
	Letting $V=(\widetilde{a_1},\widetilde{b_1}+1)$ and $g_V$ be a continuous function with $\sigma(g_V)=V$, then $\perp$ and $\perpp$ still coincide in $\mathcal{A}\cup\left\{g_V\right\}$, however $\subseteq$ and $\Subset$ do not, since $f_1\subseteq g_V$ but not $f_1\Subset g_V$.
	
	Alternatively, set $W=(\widetilde{b_1},\widetilde{b_1}+1)$ and let $g_W$ be any continuous function with $\sigma(g_W)=W$. Then $\subseteq$ and $\Subset$ still coincide in $\mathcal{A}\cup\left\{g_W\right\}$, however $\perp$ and $\perpp$ do not, because $f_1\perp g$ but not $f_1\perpp g$.
\end{example}

\subsection{\texorpdfstring{$\perpp$}{⊥⊥}-ideals}

Recall that $X$, $H$ and $\theta\in C(X,H)$ are fixed, as in the beginning of the \hyperref[sectiondisjointness]{section}. We fix also a weakly regular family $\mathcal{A}\subseteq C_c(X,\theta)$.

One technique that is commonly used in the proofs of Banach--Stone type theorems is to describe an order-isomorphism between $X$ and ``maximal ideals'' of $\mathcal{A}$, where the notion of an ``ideal'' depends on whatever kinds of algebraic signature one is working. See for example \cite[Lemma 2.2]{MR0029476}, \cite[Proposition 2.7]{MR2324919}, \cite[Lemma 3]{MR0020715}. This idea also appears in some manner in the proofs of the main results of \cite{MR3162258} and \cite{MR1060366}, and of \cite[p. 170, Théorème 3]{MR1357166}.

We will follow this idea by considering \emph{$\perpp$-ideals}. Although their definition (\ref{definitionperppideal}) is given simply in terms of the relation $\perpp$, Theorem \ref{theoremperppideals} provides a much more manageable description of them. 

A strengthening of the notion of cover will be necessary (see Lemma \ref{lemmagoodformofstrongcover} for the intuition).

\begin{definition}\label{definitionstrongcover}\index{Cover!strong}
	A finite family $B\subseteq\mathcal{A}$ is said to be a \emph{strong cover} of an element $a\in\mathcal{A}$ if there is another finite family $\widetilde{B}\subseteq\mathcal{A}$ such that:
	\begin{enumerate}[label=(SC\arabic*)]
		\item For all $\widetilde{b}\in\widetilde{B}$, there is some $b\in B$ with $\widetilde{b}\Subset b$;
		\item $\widetilde{B}$ is a cover of $a$ (see Definition \ref{definitiondisjointcover}).
	\end{enumerate}
\end{definition}

The following proposition shows that strong covers encode information about the closures of sets. It is a direct consequence of the definition of $\Subset$ and Lemma \ref{lemmacovers}.

\begin{lemma}\label{lemmagoodformofstrongcover}
	A finite family $B\subseteq\mathcal{A}$ is a strong cover of $a\in\mathcal{A}$ if and only if $\supp(a)\subseteq\bigcup_{b\in B}\sigma(b)$.
\end{lemma}

\begin{definition}\index{Ideal!$\perpp$-ideal}\label{definitionperppideal}
	A \emph{$\perpp$-ideal} in $\mathcal{A}$ is a subset $I\subseteq\mathcal{A}$ such that, for all $a\in\mathcal{A}$,
	\[a\in I\iff \text{there is a finite subset }B\subseteq I\text{ which is a strong cover of }a.\]
\end{definition}

Note that every $\perpp$-ideal of $\mathcal{A}$ contains $\theta$ (since the empty set is a strong cover of $\theta$).

We will now prove that the lattice of open subsets of a space $X$ is order-isomorphic to the lattice of $\perpp$-ideals of a weakly regular tuple $(X,\theta,\mathcal{A})$.

\begin{definition}
	Suppose $(X,\theta,\mathcal{A})$ is weakly regular. Given an open set $U\subseteq X$, denote $\mathbf{I}(U)=\left\{f\in\mathcal{A}:\supp(f)\subseteq U\right\}$, and , given a $\perpp$-ideal $I\subseteq\mathcal{A}$, denote $\mathbf{U}(I)=\bigcup_{f\in I}\sigma(f)$.
\end{definition}

Lemma \ref{lemmagoodformofstrongcover} and weak regularity of $\mathcal{A}$ imply that $\mathbf{I}(U)$ is a $\perpp$-ideal of $\mathcal{A}$ for any open $U\subseteq X$.

\begin{theorem}\label{theoremperppideals}
	Suppose $(X,\theta,\mathcal{A})$ is weakly regular.
	\begin{enumerate}[label=(\alph*)]
		\item\label{theoremperppideals(a)} For every $\perpp$-ideal $I$ of $\mathcal{A}$, $I=\mathbf{I}(\mathbf{U}(I))$;
		\item\label{theoremperppideals(b)} For every open subset $U\subseteq X$, $U=\mathbf{U}(\mathbf{I}(U))$;
		\item\label{theoremperppideals(c)} The map $U\mapsto\mathbf{I}(U)$ is an order isomorphism between the lattices of open sets of $X$ and $\perpp$-ideals of $\mathcal{A}$.
	\end{enumerate}
\end{theorem}
\begin{proof}
	\begin{enumerate}[label=(\alph*)]
		\item Let $I$ be a $\perpp$-ideal. The inclusion $I\subseteq \mathbf{I}(\mathbf{U}(I))$ follows easily from the definition of $\perpp$-ideals: if $f\in I$, then take a finite strong cover $B\subseteq I$ of $f$, so that $\supp(f)\subseteq\bigcup_{b\in B}\sigma(b)\subseteq\mathbf{U}(I)$.
		
		Conversely, if $f\in\mathbf{I}(\mathbf{U}(I))$ then $\supp(f)\subseteq\mathbf{U}(I)=\bigcup_{b\in I}\sigma(b)$. Using compactness of $\supp(f)$ we find a finite family $B\subseteq I$ with $\supp(f)\subseteq\bigcup_{b\in B}\sigma(b)$, so $B$ is a strong cover of $f$ (Lemma \ref{lemmagoodformofstrongcover}) and therefore $f\in I$.
		\item Suppose $U\subseteq X$ is open. By weak regularity of $(X,\theta,\mathcal{A})$ and since $X$ is regular, we have
		\[U=\bigcup_{\substack{f\in\mathcal{A}\\\supp(f)\subseteq U}}\sigma(f)=\bigcup_{f\in\mathbf{I}(U)}\sigma(f)=\mathbf{U}(\mathbf{I}(U)).\]
		\item The previous items prove that $U\mapsto\mathbf{I}(U)$ is a bijection, with inverse $I\mapsto\mathbf{U}(I)$. It is clear that both maps are order-preserving.\qedhere
	\end{enumerate}
\end{proof}

\subsection{The main theorems}
The main theorem (\ref{maintheorem}) now follows easily from the previous subsection. Fix two locally compact Hausdorff spaces $X$ and $Y$, and for $Z\in\left\{X,Y\right\}$ a Hausdorff space $H_Z$, a continuous map $\theta_Z\colon Z\to H_Z$, and a subset $\mathcal{A}(Z)\subseteq C_c(Z,\theta_Z)$.

\begin{definition}
	We call a map $T\colon\mathcal{A}(X)\to\mathcal{A}(Y)$ a \emph{$\perpp$-morphism} if $f\perpp g$ implies $Tf\perpp Tg$; $T$ is a \emph{$\perpp$-isomorphism} if it is bijective and both $T$ and $T^{-1}$ are $\perpp$-morphisms. $\perp$, $\subseteq$ and $\Subset$-isomorphisms are define analogously.\index{$\perpp$-morphism}\index{$\perp$-morphism}\index{$\subseteq$-morphism}\index{$\Subset$-morphism}
\end{definition}

By Proposition \ref{propositionrelationsrelations}\ref{theoremrelationsrelations(a)}, $\perp$-morphisms coincide with $\subseteq$ morphisms. We obtain:

\begin{theorem}\label{theoremdisjoint}
	Suppose $(X,\theta_X,\mathcal{A}(X))$ and $(Y,\theta_Y,\mathcal{A}(Y))$ are weakly regular and $T\colon\mathcal{A}(X)\to\mathcal{A}(Y)$ is a $\perp$-isomorphism. Let $f,g\in\mathcal{A}(X)$. Then $\sigma(f)\subseteq \sigma(g)$ if and only if $\sigma(Tf)\subseteq \sigma(Tg)$. In particular, $Z(f)=\varnothing$ if and only if $Z(Tf)=\varnothing$.
\end{theorem}

Assume $(X,\theta,\mathcal{A}(X))$ is weakly regular. Let $\widehat{\mathcal{A}(X)}$ be the collection of maximal $\perpp$-ideals of $\mathcal{A}$, and endow it with the topology generated by the sets
\[U(f)=\left\{I\in\widehat{\mathcal{A}(X)}:\exists g\Subset f\text{ such that }g\not\in I\right\},\qquad f\in\mathcal{A}(X).\]
By Theorem \ref{theoremperppideals}, we obtain a bijection $\kappa_X\colon X\to\widehat{\mathcal{A}(X)}$, $\kappa_X(x)=\mathbf{I}(X\setminus\left\{x\right\})$. Since for all $x\in X$ and $f\in\mathcal{A}(X)$,
\[x\in\sigma(f)\iff\exists g\Subset f\text{ such that } x\in\supp(g)\iff\kappa_X(x)\in U(f),\]
then $\kappa_X(\sigma(f))=U(f)$, which proves that $\kappa_X$ is a homeomorphism. Performing a similar procedure  with $Y$ and using standard duality arguments, we obtain our main theorem:

\begin{theorem}\label{maintheorem}
	If $\mathcal{A}(X)$ and $\mathcal{A}(Y)$ are weakly regular and $T\colon\mathcal{A}(X)\to\mathcal{A}(Y)$ is a $\perpp$-isomorphism then there is a unique homeomorphism $\phi\colon Y\to X$ such that $\phi(\supp(Tf))=\supp(f)$ for all $f\in \mathcal{A}(X)$ (equivalently, $\phi(\sigma(Tf))=\sigma(f)$, or $\phi(Z(Tf))=Z(f)$, for all $f\in\mathcal{A}(X)$).
\end{theorem}

\begin{definition}\label{definitionthomeomorphism}\index{$T$-homeomorphism}
	The unique homeomorphism $\phi$ associated with $T$ as in \ref{maintheorem} will be called the \emph{$T$-homeomorphism}.
\end{definition}

We finish this section by proving that Theorem \ref{maintheorem} is sharp, in the sense that the analogous result for $\perp$-isomorphisms is not true in general. We can even be more bold and find counter-examples in settings which are usually regarded as ``well-behaved''; namely, we will consider only real-valued functions and the usual notion of support (i.e., $C_c(X)=C_c(X,0)$ for a space $X$), and compact spaces.

Moreover, we will provide two examples: one where the underlying topological spaces do not have the same small inductive dimension (Corollary \ref{corollary01stoneperp}), and one where they do (Corollary \ref{corollary01circleperp}).

Let us fix some notation, and recall some basic facts about Stone duality. We refer to \cite{MR1507106,dml124080} for details (see also \cite[II.4.4]{MR861951})

\begin{named}{Notation}
	Given a Hausdorff space $X$, denote by $\operatorname{RO}_K(X)$ the generalized Boolean algebra of regular open subsets of $X$ with compact closure, and by $\operatorname{KO}(X)$ the generalized Boolean algebra of compact-open subsets of $X$. Given $A\in \operatorname{RO}_K(X)$, we define $\Sigma_X(A)=\left\{f\in C_c(X):\sigma(f)=A\right\}$.
	
	Given a generalized Boolean algebra $B$, let $\operatorname{Spec}(B)$ be the spectrum of $B$, i.e., the space of non-trivial lattice homomorphisms from $B$ to the two-element lattice $\left\{0,1\right\}$, with the topology of pointwise convergence. (Equivalently, it may be regarded as the space of ultrafilters of $B$.)
\end{named}

\begin{named}{Stone duality}
	The usual form of Stone duality states that the category of Stone (i.e., zero-dimensional, compact Hausdorff) spaces is dual to that of Boolean algebras. This extends to the \emph{locally} compact spaces and \emph{generalized} Boolean algebras, and in particular we obtain: Every zero-dimensional, locally compact Hausdorff space $X$ is (naturally) homeomorphic to $\operatorname{Spec}(\operatorname{KO}(X))$, and every generalized Boolean algebra $B$ is (naturally) isomorphic to $\operatorname{KO}(\operatorname{Spec}(B))$. For a more general version, see \cite{bicestarlinglcsd}.
\end{named}

In order to find non-homeomorphic spaces $X$ and $Y$ such that $C_c(X)$ and $C_c(Y)$ are $\perp$-isomorphic, we need the following result:

\begin{proposition}\label{propositionperpisnotenough}
	Suppose that:
	\begin{enumerate}[label=(\roman*)]
		\item\label{theoremperpisnotenough(i)} $X$ and $Y$ are separable, locally compact Hausdorff spaces;
		\item\label{theoremperpisnotenough(ii)} For all nonempty $A\in\operatorname{RO}_K(X)$ and $B\in\operatorname{RO}_K(Y)$, $\Sigma_X(A)$ and $\Sigma_Y(B)$ are nonempty;
		\item\label{theoremperpisnotenough(iii)} $\varphi\colon\operatorname{RO}_K(X)\to\operatorname{RO}_K(Y)$ is an order isomorphism (with respect to set inclusion).
	\end{enumerate}
	Then $C_c(X)$ and $C_c(Y)$ are $\perp$-isomorphic.
\end{proposition}
\begin{proof}
	Given $A\in\operatorname{RO}_K(X)$, the sets  $\Sigma_X(A)$ and $\Sigma_Y(\varphi(A))$ have the same cardinality: They are either singletons if $A=\varnothing$, or have cardinality $2^{\aleph_0}$ otherwise, by \ref{theoremperpisnotenough(ii)} and since $X$ and $Y$ are separable. Consider any bijection $T_A\colon\Sigma_X(A)\to\Sigma_Y(\varphi(A))$. Then the map
	\[T\colon C_c(X)\to C_c(Y),\qquad T(f)=T_{\sigma(f)}(f)\]
	is a $\perp$-isomorphism.\qedhere
\end{proof}

The following are two technical lemmas which will allow us to construct spaces $X$ and $Y$ satisfying the hypotheses of the theorem above.

\begin{lemma}
	\label{lemmawhenregularcompactopenareclopen}
	Suppose that $\mathfrak{C}$ is a zero-dimensional, locally compact Hausdorff space and $\operatorname{KO}(\mathfrak{C})$ is conditionally complete (i.e., every \emph{bounded} family has a supremum). Then $\operatorname{RO}_K(\mathfrak{C})=\operatorname{KO}(\mathfrak{C})$.
\end{lemma}
\begin{proof}
	The only non-trivial part is proving $\operatorname{RO}_K(\mathfrak{C})\subseteq\operatorname{KO}(\mathfrak{C})$. Given $A\in\operatorname{RO}_K(\mathfrak{C})$, the family $\left\{V\in\operatorname{KO}(\mathfrak{C}):V\subseteq A\right\}$ is bounded in $\operatorname{KO}(\mathfrak{C})$, so let $U$ be its supremum in $\operatorname{KO}(\mathfrak{C})$. As $\mathfrak{C}$ is zero-dimensional we have $A\subseteq U$. To prove the reverse inclusion, we first show that $U\setminus\overline{A}=\varnothing$.
	
	If $W\in\operatorname{KO}(\mathfrak{C})$ and $W\subseteq U\setminus\overline{A}$, then $A\subseteq  U\setminus W$, from which it follows that $U\subseteq U\setminus W$, so $W=\varnothing$. This proves that $U\setminus\overline{A}=\varnothing$, because $\mathfrak{C}$ is zero-dimensional, and so $U\subseteq\overline{A}$. However, $U$ is clopen and $A$ is regular open, which implies $A=U$.\qedhere
\end{proof}

\begin{lemma}\label{gprokseparableisseparable}
	If $X$ is a separable locally compact Hausdorff space, then $\mathfrak{C}=\operatorname{Spec}(\operatorname{RO}_K(X))$ is separable as well.
\end{lemma}
\begin{proof}
	Let $\left\{x_n:n\in\mathbb{N}\right\}$ be a countable dense subset of $X$. For each $n$, Zorn's Lemma implies that exists $\phi_n\in\mathfrak{C}$ such that $\phi_n(U)=1$ whenever $x_n\in U$. Then $\left\{\phi_n\right\}_n$ is easily seen to be dense in $\mathfrak{C}$.\qedhere
\end{proof}

\begin{lemma}\label{lemmaregularopeninsecondcountableissigma}
	If $X$ is a second-countable locally compact Hausdorff space and $A\in\operatorname{RO}_K(X)$, then there is $f\in C_c(X)$ such that $\sigma(f)=A$.
\end{lemma}
\begin{proof}
	First choose a countable family of compact subsets $K_n\subseteq A$ such that $\bigcup_nK_n=A$. For each $n$ we can, by Urysohn's Lemma and regularity of $X$, find a continuous function $f_n\colon X\to[0,1]$ such that $f_n(k)=1$ for all $k\in K_n$ and $\supp(f_n)\subseteq A$. Letting $f=\sum_{n=1}^\infty 2^{-n}f_n$ we obtain $[f\neq 0]=\sigma(f)=A$, because $A$ is regular.
\end{proof}

Given locally compact Hausdorff $X$, let $\mathfrak{C}=\operatorname{Spec}(\operatorname{RO}_K(X))$. The generalized Boolean algebra $\operatorname{RO}_K(X)$ is conditionally complete, so by Stone duality, $\operatorname{KO}(\mathfrak{C})$ is also conditionally complete, and hence coincides with $\operatorname{RO}_K(\mathfrak{C})$. As a consequence of Lemmas \ref{lemmawhenregularcompactopenareclopen}, \ref{gprokseparableisseparable} and \ref{lemmaregularopeninsecondcountableissigma} and Proposition \ref{propositionperpisnotenough} when $X=[0,1]$, we conclude:

\begin{corollary}\label{corollary01stoneperp}
	There exists a zero-dimensional, compact Hausdorff topological space $\mathfrak{C}$ (name\-ly, $\mathfrak{C}=\operatorname{Spec}(\operatorname{RO}_K([0,1]))$) -- which is, in particular, not homeomorphic to $[0,1]$ -- such that $C(\mathfrak{C})$ and $C([0,1])$ are $\perp$-isomorphic.
\end{corollary}

For our second example, we will consider only second-countable spaces. The next lemma is again a technical lemma which can be proven by elementary topological considerations, so we omit its proof.

\begin{lemma}\label{lemmamorphismregularopenofsubset}
	Let $X$ be a topological space and $U$ a dense open subset of $X$. Then the map
	\[\varphi_U\colon \operatorname{RO}(X)\to\operatorname{RO}(U),\qquad \varphi_U(A)=A\cap U\]
	is an order isomorphism.
\end{lemma}

Let $\mathbb{S}^1=\left\{z\in\mathbb{C}:|z|=1\right\}$ be the complex unit circle.

\begin{corollary}\label{corollary01circleperp}
	$C([0,1])$ and $C(\mathbb{S}^1)$ are $\perp$-isomorphic.
\end{corollary}
\begin{proof}
	Let $X=(0,1)$ and $Y=\mathbb{S}^1\setminus\left\{1\right\}$. Then $X$ and $Y$ are homeomorphic, and two applications of Lemma \ref{lemmamorphismregularopenofsubset} imply that $\operatorname{RO}([0,1])$ and $\operatorname{RO}(\mathbb{S}^1)$ are order-isomorphic. Lemmas \ref{gprokseparableisseparable} and \ref{lemmaregularopeninsecondcountableissigma}, and Proposition \ref{propositionperpisnotenough}, imply that $C([0,1])$ and $C(\mathbb{S}^1)$ are $\perp$-isomorphic.
\end{proof}

\section{Basic maps}\label{sectionbasicmaps}

In this section we will develop techniques to classify isomorphisms for spaces of functions with different algebraic structures. As in the preceding sections, we will be interested mostly in spaces of continuous functions between topological spaces, however the initial notions we will deal with can be defined in purely set-theoretical terms.

Let $X$ and $H_X$ be sets, and consider a class $\mathcal{A}(X)\subseteq(H_X)^X$ of $H_X$-valued functions on $X$. Given a point $x\in X$, denote by $\mathcal{A}(X)|_x$ the set of images of $x$ under elements of $\mathcal{A}(X)$, i.e.
\[\mathcal{A}(X)|_x=\left\{f(x):f\in\mathcal{A}(X)\right\}.\ntag\label{definitionimageofpointunderclassoffunction}\]

If $Y$ is another set and $\phi\colon Y\to X$ is a map, denote by
\[Y\times_{(\phi,\mathcal{A}(X))} H_X=\bigcup_{y\in Y}\left\{y\right\}\times\mathcal{A}(X)|_{\phi(y)}=\left\{(y,f(\phi(y))):y\in Y,f\in\mathcal{A}(X)\right\}.\]

Note that $Y\times_{(\phi,\mathcal{A}(X))} H_X$ is equal to $Y\times H_X$ if and only if the following property is satisfied: For every $y\in Y$ and every $c\in H_X$, there exists $f\in\mathcal{A}(X)$ such that $f(\phi(y))=c$.

\begin{definition}\label{definitionttransform}
	Let $X$, $H_X$, $Y$ and $H_Y$ be sets, $\phi\colon Y\to X$ be a function, and consider a class of functions $\mathcal{A}(X)\subseteq (H_X)^X$.
	
	Given maps $\phi\colon Y\to X$ and $\chi\colon Y\times_{(\phi,\mathcal{A}(X))} H_X\to H_Y$, we define $T_{(\phi,\chi)}\colon\mathcal{A}(X)\to (H_Y)^Y$ by
	\[(T_{(\phi,\chi)}f)(y)=\chi(y,f(\phi(y)),\qquad\forall f\in\mathcal{A}(X),\quad\forall y\in Y.\ntag\label{definitiontphichi}\]
\end{definition}

\begin{definition}\label{definitionbasic}
	Let $X$, $H_X$, $Y$ and $H_Y$ be sets, $\phi\colon Y\to X$ be a function, and consider classes of functions $\mathcal{A}(X)\subseteq (H_X)^X$ and $\mathcal{A}(Y)\subseteq (H_Y)^Y$. A map $T\colon\mathcal{A}(X)\to\mathcal{A}(Y)$ is called \emph{$\phi$-basic} if there exists $\chi\colon Y\times_{(\phi,\mathcal{A}(X))} H_X\to H_Y$ such that $T=T_{(\phi,\chi)}$. We call such $\chi$ a \emph{$(\phi,T)$-transform}.
	
	We denote \emph{sections} of $\chi$ by $\chi(\cdot,y)\colon \mathcal{A}(X)|_{\phi(y)}\to H_Y$ (where $y\in Y$).
\end{definition}

(In the definition above, we ignore the fact that the codomain of $T$ is $\mathcal{A}(Y)$, while the codomain of $T_{(\phi,\chi)}$ is $(H_Y)^Y$.)

In simpler terms, a basic map is one that is induced naturally by the transformation $\phi\colon Y\to X$, and the field $\left\{\chi(y,\cdot):y\in Y\right\}$ of partial functions on $H_X$.

\begin{example}
	Let $\phi\colon Y\to X$ and $\psi\colon H_X\to H_Y$ be functions. Then the map
	\[T\colon(H_X)^X\to (H_Y)^Y,\qquad Tf=\psi\circ f\circ \phi\]
	is $\phi$-basic, and the $(\phi,T)$-transform $\chi$ is given by $\chi(y,z)=\psi(z)$.
\end{example}

The next example will appear, in some form, in most applications in Section \ref{sectionconsequences}.

\begin{example}
	Suppose that $X$ is a locally compact Hausdorff space, $H_X$ is a Hausdorff space, $\theta_X\in C(X,H_X)$ and $\mathcal{A}(X)\subseteq C_c(X,\theta_X)$. Let $Y$ and $H_Y$ be topological spaces, $\phi\colon Y\to X$ be a homeomorphism, and $\chi\colon Y\times H_X\to H_Y$ be a continuous map such that for every $y\in Y$, the section $\chi(y,\cdot)\colon H_X\to H_Y$ is a bijection.
	
	For every $f\in\mathcal{A}(X)$, define $Tf\in C(Y,H_Y)$ as
	\[Tf\colon Y\to H_Y,\qquad Tf(y)=\chi(y,f(\phi(y))),\]
	and let $\mathcal{A}(Y)=\left\{Tf\colon f\in\mathcal{A}(X)\right\}$. Also define $\theta_Y=T\theta_X$. Then
	\begin{enumerate}[label=\arabic*.]
		\item $T$ is $\phi$-basic, and the $(\phi,T)$-transform is the restriction of $\chi$ to $Y\times_{(\phi,\mathcal{A}(X))} H_X$;
		\item $(X,\theta_X,\mathcal{A}(X))$ is (weakly) regular if and only if $(Y,\theta_Y,\mathcal{A}(Y))$ is (weakly) regular. In this case, $T$ is a $\perpp$-isomorphism and $\phi$ is the $T$-homeomorphism.
	\end{enumerate}
\end{example}

Note that not every $\perpp$-isomorphism is given as in the previous example.

\begin{example}
	Suppose that $X=Y$ is compact Hausdorff, $H_X=H_Y=\mathbb{R}$ and $\theta_X=\theta_Y=0$, so we simply write $C(X)=C(X,\mathbb{R})$. Let $T\colon C(X)\to C(X)$ be any bijection satisfying $[f\neq 0]=[Tf\neq 0]$ for all $f\in C(X)$. Then $T$ is a $\perpp$-isomorphism, and the $T$-homeomorphism is the identity $\id_X\colon X\to X$. Let us look at two particular cases:
	\begin{itemize}
		\item Suppose that $X$ is not a singleton, $T(1)=2$, $T(2)=1$, and $Tf=f$ for every $f\neq 1,2$. Then $T$ is non-basic (see Proposition \ref{propositionbasicisomorphisms}\ref{propositionbasicisomorphisms(a)}) and discontinuous with respect to either the topology of uniform convergence, or the topology of pointwise convergence.
		
		\item If $X=\{\ast\}$ is a singleton, we identify $C(X)$ with $\mathbb{R}$, so any self-bijection $T\colon\mathbb{R}\to\mathbb{R}$ preserving $0$ is a basic $\perpp$-automorphism. In this case, the $T$-transform $\chi$ coincides with $T$ (or more precisely $\chi(\ast,z)=T(z)$ for all $z\in\mathbb{R}$), and most (cardinality-wise) of these are discontinuous: indeed, there are $2^{2^\aleph_0}\mathfrak{c}$ self-bijections of $\mathbb{R}\setminus\left\{0\right\}$, but only $2^{\aleph_0}$ of these are continuous.
	\end{itemize}
\end{example}

In the next proposition, we again consider only sets (without topologies).

\begin{proposition}\label{propositionbasicisomorphisms}
	Let $\mathcal{A}(X)\subseteq (H_X)^X$ and $\mathcal{A}(Y)\subseteq (H_Y)^Y$, and consider maps $\phi\colon Y\to X$ and $T\colon\mathcal{A}(X)\to\mathcal{A}(Y)$. Then
	\begin{enumerate}[label=(\alph*)]
		\item\label{propositionbasicisomorphisms(a)} $T$ is $\phi$-basic if and only if for all $y\in Y$, the following implication holds:
		\[f(\phi(y))=g(\phi(y))\Longrightarrow Tf(y)=Tg(y),\qquad\forall f,g\in\mathcal{A}(X).\ntag\label{equationpropositionbasicisomorphisms}\]
	\end{enumerate}
	In this case,
	\begin{enumerate}[label=(\alph*)]\setcounter{enumi}{1}
		\item\label{propositionbasicisomorphisms(b)} the $(\phi,T)$-transform $\chi$ is unique.
		\item\label{propositionbasicisomorphisms(c)} A section $\chi(y,\cdot)$ is injective if and only if
		\[Tf(y)=Tg(y)\Longrightarrow f(\phi(y))=g(\phi(y)),\qquad\forall f,g\in\mathcal{A}(X).\ntag\label{equationpropositionbasicisomorphisms2}\]
		\item\label{propositionbasicisomorphisms(d)} A section $\chi(y,\cdot)$ is surjective if and only if $H_Y=\left\{Tf(y):f\in\mathcal{A}(X)\right\}$.
	\end{enumerate}
\end{proposition}
\begin{proof}
	For one direction of \ref{propositionbasicisomorphisms(a)}, if $T$ satisfies \eqref{equationpropositionbasicisomorphisms}, define $\chi$ by
	\[\chi(y,t)=Tf(y),\]
	whenever $f\in\mathcal{A}(X)$ is any function satisfying $f(\phi(y))=t$. Then $\chi(y,t)$ does not depend on the choice of $f$ by implication \eqref{equationpropositionbasicisomorphisms}, and hence $T$ is $\phi$-basic.
	
	The converse direction of \ref{propositionbasicisomorphisms(a)}, as well as items \ref{propositionbasicisomorphisms(b)}, \ref{propositionbasicisomorphisms(c)} and \ref{propositionbasicisomorphisms(d)} are immediate from the formula $Tf(y)=\chi(y,f(\phi(y)))$, which holds for all $f\in\mathcal{A}(X)$ and $y\in Y$.
\end{proof}

In our applications, we will use item \ref{propositionbasicisomorphisms(a)} above several times, by proving that a given $\perpp$-isomorphism $T\colon\mathcal{A}(X)\to\mathcal{A}(Y)$, between regular classes of functions, is basic with respect to the associated homeomorphism $\phi\colon Y\to X$ given by Theorem \ref{maintheorem}. This is, in fact, the only possibility: If $\psi\colon Y\to X$ is any map such that $T$ is $\psi$-basic, then $\psi$ is the $T$-homeomorphism! Since this fact will not be used later in the article we refer its proof to \cite[Proposition 3.3.7]{cordeirothesis}.

\subsection{Algebraic signatures and basic maps}

In the \hyperref[sectionconsequences]{next section} we will consider different algebraic structures on spaces of continuous functions. To this end, recall (see \cite{MR1221741}) that an \emph{algebraic signature}\index{Signature!Algebraic} is a collection $\eta$ of pairs $(*,n)$, where $*$  is a (function) symbol and $n$ is a non-negative integer, called the \emph{arity} of $*$. A \emph{model} of $\eta$ consists of a set $H$ and a map associating to each $(*,n)\in\eta$ a function $*\colon H^n\to H$, $(c_1,\ldots,c_n)\mapsto c_1*\cdots*c_n$. (We use the convention that $H^0$ is a singleton set, so that a $0$-ary function symbol is the same as a constant.)

For example, the usual signature of groups consists of one binary symbol $\cdot$ (for the product), one unary symbol $(\ )^{-1}$ (the inversion) and one constant/0-ary symbol $1$ (the unit).

If $H$ is a model of $\eta$ and $X$ is a set then the function space $H^X$ can also be regarded as a model of $\eta$ with the pointwise structure: $(f_1*\cdots*f_n)(x)=f_1(x)*\cdots*f_n(x)$ for all $f_1,\ldots,f_n\in H^X$, all $x\in X$, and all $n$-ary function symbols $\ast$.

A \emph{morphism} of two models $H_1$ and $H_2$ of a given signature $\eta$ is a map $m\colon H_1\to H_2$ such that for any $n$-ary function symbol $*$ of $\eta$ and any $x_1,\ldots,x_n\in H_1$, we have $m(x_1*\cdots*x_n)=m(x_1)*\cdots*m(x_n)$.

Finally, a \emph{submodel} of a model $H$ of a signature $\eta$ is a subset $K\subseteq H$ such that for all $n$-ary symbols $*$ of $\eta$ and any $d_1,\ldots,d_n\in K$, $d_1*\cdots*d_n\in K$, so that $K$ can be naturally regarded as a model of $\eta$.

In the topological setting, a \emph{continuous model} $H$ of a signature $\eta$ is defined in the same manner, but we assume that all maps are continuous. In this case, if $X$ is a topological space then $C(X,H)$ is a submodel of $H^X$.

\begin{proposition}\label{propositionmodelmorphism}
	Let $X$ and $Y$ be sets. Suppose that $H_X$ and $H_Y$ are models for an algebraic signature $\eta$, and that $\mathcal{A}(X)$ and $\mathcal{A}(Y)$ are submodels of $(H_X)^X$ and $(H_Y)^Y$. Then for all $x\in X$, $\mathcal{A}(X)|_x$ is a submodel of $H_X$, and similarly $\mathcal{A}(Y)|_y$ is a submodel of $H_Y$ for all $y\in Y$. (See Equation \eqref{definitionimageofpointunderclassoffunction}.)
	
	Let $T\colon\mathcal{A}(X)\to\mathcal{A}(Y)$ be a basic map with respect to a function $\phi\colon Y\to X$, and let $\chi$ be the $T$-transform. Then $T$ is a morphism (for $\eta$) if and only if every section $\chi(y,\cdot)$ is a morphism (from $\mathcal{A}(X)|_{\phi(y)}$ to $\mathcal{A}(Y)|_y$).
\end{proposition}
\begin{proof}
	Given $x\in X$, the evaluation map $\pi_x\colon(H_X)^X\to H_X$, $\pi_x(f)=f(x)$, is a morphism, and it follows that $\mathcal{A}(X)|_x=\pi_x(\mathcal{A}(X))$ is a submodel of $H_X$.
	
	Note that for all $y\in Y$, $\chi(y,\cdot)\circ\pi|_{\phi(y)}=\pi_y\circ T$. On one hand, $T$ is a morphism if and only if $\pi_y\circ T$ is a morphism for all $y$. On the other, $\pi|_{\phi(y)}$ is a surjective morphism from $\mathcal{A}(X)$ to $\mathcal{A}(X)|_{\phi(y)}$. It follows that $T$ is a morphism if and only if $\chi(y,\cdot)$ is a morphism for all $y\in Y$.\qedhere
\end{proof}

\subsection{Group-valued maps}\label{subsectiongroupvaluedmaps}

In several applications, we will consider groups of functions, and in this case a slight, but nevertheless important, simplification of Proposition \ref{propositionbasicisomorphisms}\ref{propositionbasicisomorphisms(a)} will be used.

\begin{proposition}\label{propositionbasicnessofgroupvalued}
	Suppose that $H_X$ and $H_Y$ are groups, $\mathcal{A}(X)$ and $\mathcal{A}(Y)$ are subgroups of $(H_X)^X$ and $(H_Y)^Y$, respectively, $T\colon\mathcal{A}(X)\to\mathcal{A}(Y)$ is a group isomorphism and $\phi\colon Y\to X$ is a function. Then $T$ is $\phi$-basic if and only if for all $y\in Y$,
	\[f(\phi(y))=1\Longrightarrow Tf(y)=1,\qquad\forall f\in\mathcal{A}(X).\]
\end{proposition}

Assume that $H$ is a topological group and $X$ is a locally compact Hausdorff space. Any subgroup $\mathcal{A}$ of $C(X,H)$ contains the constant function $1$, and if $\theta\in\mathcal{A}$, then the map $f\mapsto f\theta^{-1}$ is a $\perpp$-isomorphism between $(X,\theta,\mathcal{A})$ and $(X,1,\mathcal{A})$. In this case, since we will be mostly interested in group isomorphisms between subgroups of $C(X,H)$, we may always assume that $\theta=1$. In the case that $H=\mathbb{R}$ or $\mathbb{C}$, as additive groups, we recover the usual notion of support.

\subsection{Continuity}

Now, we study continuity of basic $\perpp$-isomorphisms and relate it to the continuity of its transform. For this, it is necessary to construct functions which attain predetermined values on infinitely many points (e.g.\ the points of some converging net). One procedure for this is by ``cutting and pasting'' continuous functions, although this sometimes requires some first countability or connectedness hypotheses in order to maintain control of the final function. This technique is somewhat elementary, although one needs to take some care in order to guarantee continuity, so we refer the most cumbersome details to \cite[Section 3.3.3]{cordeirothesis}, and sketch the main points of the proofs. The next proposition can be proven by elementary Topology.

\begin{proposition}\label{propositiondisjointopensets}
	If $F$ is an infinite subset of a regular Hausdorff space $X$, then there exists a countable infinite subset $\left\{y_1,y_2,\ldots\right\}\subseteq F$ and pairwise disjoint open sets $U_n$ such that $y_n\in U_n$ for all $n$.
\end{proposition}

For the next proposition, recall (\cite[27.4]{MR0264581}) that a topological space $H$ is \emph{locally path-connected} if every point $t\in H$ admits a neighbourhood basis consisting of path-connected subsets.

\begin{proposition}\label{propositionconstructfunction}
	Let $X$ be a locally compact Hausdorff space, $\left\{x_n\right\}_n$ be a sequence of elements of $X$, $\left\{U_n\right\}_n$ a sequence of pairwise disjoint open subsets of $X$ with $x_n\in U_n$ for all $n$.
	
	Let $H$ be a Hausdorff first-countable locally path-connected topological space and consider a family $\left\{g_n\colon U_n\to H\right\}_n$ of continuous functions such that $g_n(x_n)$ converges to some $t\in H$. Then
	\begin{enumerate}[label=(\alph*)]
		\item\label{propositionconstructfunction(a)} there exists a continuous function $f\colon X\to H$ such that $f(x_n)=g_n(x_n)$ for all sufficiently large $n$, and $f(x)=t$ for all $x\not\in\bigcup_n U_n$.
		\item\label{propositionconstructfunction(b)} if $H=\mathbb{R}$, there is a continuous function $f\colon X\to\mathbb{R}$ such that $f=g_n$ on a neighbourhood of $x_n$ and $f(x)=t$ for all $x\not\in\bigcup_n U_n$.
	\end{enumerate}
\end{proposition}
\begin{proof} Item \ref{propositionconstructfunction(b)} is an easy application of Tietze's Extension Theorem, so we concentrate on item \ref{propositionconstructfunction(a)}. Let $\left\{W_n\right\}_n$ be a decreasing basis of path-connected neighbourhoods of $t$. Disregarding any $n$ such that $g_n(x_n)$ does not belong to $W_1$, and repeating the sets $W_k$ if necessary (i.e., considering a new sequence of neighbourhoods of $t$ of the form
	\[W_1,W_1,\ldots,W_1,W_2,W_2,\ldots,W_2,\ldots,\]
	where each $W_k$ is repeated finitely many times) we may assume that $t_n:=g_n(x_n)\in W_n$.
	
	For each $n$, take a continuous path $\alpha_n\colon [0,1]\to W_n$ such that $\alpha_n(0)=t_n$ and $\alpha_n(1)=t$. Now take continuous functions $b_n\colon X\to[0,1]$ such that $b_n(x_n)=0$ and $b_n=1$ outside $U_n$. Define $f$ as $\alpha_n\circ b_n$ on each $U_n$, and as $t$ on $X\setminus\bigcup_n U_n$.
	
	The only non-trivial part about continuity of $f$ is proving that $f$ is continuous on the boundary $\partial\left(\bigcup_n U_n\right)$. If $x$ belongs to this set then $f(x)=t$. Given a basic neighbourhood $W_N$ of $t$, we have that $\bigcap_{n=1}^N(\alpha_n\circ b_n)^{-1}(W_N)$ is a neighbourhood of $x$ contained in $f^{-1}(W_N)$, and thus $f$ is continuous. Item \ref{propositionconstructfunction(b)} uses similar arguments.
\end{proof}

\begin{theorem}\label{theorembasicisomorphisms}
	Let $X$ and $Y$ be locally compact Hausdorff and for $Z\in\left\{X,Y\right\}$, $H_Z$ a Hausdorff space and $\theta_Z\in C(Z,H_Z)$ be given such that $(Z,\theta_Z,C_c(Z,\theta_Z))$ is regular.
	
	Suppose that $T\colon C_c(X,\theta_X)\to C_c(Y,\theta_Y)$ is a $\perpp$-isomorphism, that $\phi\colon Y\to X$ is the $T$-homeomorphism $\phi$, and that that $T$ is $\phi$-basic. Let $\chi\colon Y\times H_X\to H_Y$ be the corresponding $(\phi,T)$-transform. Consider the following statements:
	\begin{enumerate}[label=(\roman*)]
		\item\label{theorembasicisomorphisms(1)} $\chi$ is continuous.
		\item\label{theorembasicisomorphisms(2)} Each section $\chi(y,\cdot)$ is a continuous;
		\item\label{theorembasicisomorphisms(3)} $T$ is continuous with respect to the topologies of pointwise convergence.
	\end{enumerate}
	Then the implications \ref{theorembasicisomorphisms(1)}$\Rightarrow$\ref{theorembasicisomorphisms(2)}$\iff$\ref{theorembasicisomorphisms(3)} always hold.
	
	If $X$, $Y$ and $H_X$ are first countable, $H_X$ is locally path-connected and $\theta_X$ is constant, then $\ref{theorembasicisomorphisms(2)}\Rightarrow \ref{theorembasicisomorphisms(1)}$.
\end{theorem}

\begin{named}{Remarks}
	\begin{enumerate}[label=(\arabic*)]
		\item In the last part of the theorem, if $H_X$ admits any structure of topological group then the condition that $\theta_X$ is constant can be dropped, since we may simply compose $T$ with the $\perpp$-isomorphism $f\mapsto f\theta^{-1}$.
		\item The domain of the $(\phi,T)$-transform $\chi$ is $Y\times H_X$ because we assume that $C_c(X,\theta_X)$ is regular.
	\end{enumerate}
\end{named}

\begin{proof}
	The implication \ref{theorembasicisomorphisms(1)}$\Rightarrow$\ref{theorembasicisomorphisms(2)} is trivial.
	
	\begin{description}
		\item[\ref{theorembasicisomorphisms(2)}$\Rightarrow$\ref{theorembasicisomorphisms(3)}] Suppose $f_i\to f$ pointwise. Then for all $y$, the section $\chi(y,\cdot)$ is continuous, thus
		\[Tf_i(y)=\chi(y,f_i(\phi(y)))\to\chi(y,f(\phi(y)))=Tf(y).\]
		This proves that $Tf_i\to Tf$ pointwise.
		\item[\ref{theorembasicisomorphisms(3)}$\Rightarrow$\ref{theorembasicisomorphisms(2)}] Assume that $T$ is continuous with respect to pointwise convergence. Let $y\in Y$ be fixed. Suppose that $t_i\to t$ in $H_X$, and let us prove that $\chi(y,t_i)\to\chi(y,t)$. Choose any function $f\in C_c(X,\theta_X)$ such that $f(\phi(y))=t$.
		
		Let $\operatorname{Fin}(X)$ be the collection of finite subsets of $X$, ordered by inclusion. For every $F\in\operatorname{Fin}(X)$ and every $i\in I$, regularity of $C_c(X,\theta_X)$ allows us to construct $f_{(F,i)}\in C_c(X,\theta_X)$ such that
		\begin{enumerate}[label=(\roman*)]
			\item $f_{(F,i)}(x)=f(x)$ pointwise if $x\in F$ and $x\neq \phi(y)$; and
			\item $f_{(F,i)}(\phi(y))=t_i$.
		\end{enumerate}
		Ordering $\operatorname{Fin}(X)$ by set inclusion, we have that $f_{(F,i)}\to f$ pointwise as $(F,i)\to\infty$, so $Tf_{(F,i)}\to Tf$ pointwise as well. For each $F\in\operatorname{Fin}(X)$ and $i\in I$, we have
		\[\chi(y,t_i)=\chi(y,f_{(F,i)}(\phi(y)))=Tf_{(F,i)}(y)\]
		so by considering $i$ and $F$ sufficiently large we see that $\chi(y,t_i)\to Tf(y)=\chi(y,t)$ as $i\to\infty$.
	\end{description}
	
	We now assume further that $X$, $Y$ and $H_X$ are first countable, $H_X$ is locally path-connected and $\theta_X$ is constant. Let $c\in H_X$ such that $\theta_X(x)=c$ for all $x\in X$.
	\begin{description}
		\item[\ref{theorembasicisomorphisms(2)}$\Rightarrow$\ref{theorembasicisomorphisms(1)}] Assume that each section $\chi(y,\cdot)$ is continuous. In order to prove that $\chi$ is continuous, we simply need to prove that for any converging sequence $(y_n,t_n)\to (y,t)$ in $Y\times H_X$, we can take a subsequence $(y_{n'},t_{n'})$ such that $\chi(y_{n'},t_{n'})\to \chi(y,t)$ as $n'\to\infty$.
		
		Given a converging sequence $(y_n,t_n)\to(y,t)$, consider an open $Y'\subseteq Y$ with compact closure such that $y,y_n\in Y'$ for all $n$.
		
		We have two cases: If for a given $z\in Y$ the set $N(z)=\left\{n\in\mathbb{N}:y_n=z\right\}$ is infinite, then we necessarily have $z=y$. Restricting the sequence $(y_n,t_n)$ to $N(y)$ and using continuity of the section $\chi(y,\cdot)$, we obtain $\chi(y_n,t_n)=\chi(y,t_n)\to \chi(y,t)$ as $n\to\infty$, $n\in N(y)$.
		
		Now assume that none of the sets $N(z)=\left\{n\in\mathbb{N}: y_n=z\right\}$ ($z\in Y$) is infinite. We may then take a subsequence and assume that all the elements $y_n$ are distinct, and actually never equal to $y$. Using Propositions \ref{propositiondisjointopensets} and \ref{propositionconstructfunction}\ref{propositionconstructfunction(a)}, and taking another subsequence if necessary, we find a continuous function $f\colon\overline{\phi(Y')}\to H_X$ such that $f(\phi(y_n))=t_n$ and $f=t$ on $X\setminus \bigcup_n U_n$. In particular, $f=t$ on the boundary $\partial(\phi(Y'))$.
		
		We now need to extend $f$ to an element of $C_c(X,\theta_X)$ (this is where we use that $\theta_X=c$ is constant). We have two cases:
		\begin{description}
			\item[Case 1] $t$ is in the path-connected component of $c$:
			
			Since $H_X$ is locally path-connected, there is a continuous path $\beta\colon[0,1]\to H_X$ with $\beta(0)=t$ and $\beta(1)=c$. Let $g\colon X\to[0,1]$ be a function with $g=0$ on $\phi(Y')$ and $g=1$ outside of a compact containing $\phi(Y')$. By defining $f=\beta\circ g$ outside of $\phi(Y')$, we obtain $f\in C_c(X,\theta_X)$. ($f$ is continuous because $f=t=\beta\circ g$ on $\partial(\phi(Y')$.)
			\item[Case 2] $t$ is not in the path-connected component of $c$:
			
			Since $H_X$ is locally path-connected, its path-connected components are clopen, and regularity of $C_c(X,\theta_X)$ then implies that $X$ (and thus also $Y=\phi(X)$) is zero-dimensional. In particular, we could have assumed at the beginning that $Y'$ is clopen, so simply set $f=c$ outside of $\phi(Y')$.
		\end{description}
		
		In any case, we obtain $f\in C_c(X,\theta_X)$ with $f(\phi(y))=t$ and $f(\phi(y_n))=t_n$, so
		\[\chi(y_n,t_n)=Tf(y_n)\to Tf(y)=\chi(y,t).\qedhere\]
	\end{description}
\end{proof}

\subsection{Non-vanishing bijections}

Let $X$ and $Y$ be compact Hausdorff spaces, $H_X$ and $H_Y$ Hausdorff spaces, $\theta_X\in C(X,H_X)$, $\theta_Y\in C(Y,\theta_Y)$ and $\mathcal{A}(X)$ and $\mathcal{A}(Y)$ regular subsets of $C_c(X,\theta_X)$ and $C_c(Y,\theta_Y)$, respectively.

\begin{definition}[\cite{MR2324919}]\label{definitionnonvanishingbijection}\index{Non-vanishing}
	We call a bijection $T\colon\mathcal{A}(X)\to\mathcal{A}(Y)$ \emph{non-vanishing} if for every $f_1,\ldots,f_n\in\mathcal{A}(X)$,
	\[\bigcap_{i=1}^n[f_i=\theta_X]=\varnothing\iff\bigcap_{i=1}^n[Tf_i=\theta_Y]=\varnothing.\]
\end{definition}

\begin{proposition}\label{propositionnonvanishingbijectionisperppisomorphism}
	If $T\colon\mathcal{A}(X)\to\mathcal{A}(Y)$ is a non-vanishing bijection, then $T$ is a $\perpp$-isomorphism.
\end{proposition}
\begin{proof}
	First note that $f\perp g$ if and only if $[f=\theta_X]\cup[g=\theta_X]=X$, or equivalently if every closed subset of $X$ intersects $[f=\theta_X]$ or $[g=\theta_X]$.
	
	As the sets $[h=\theta_X]$ ($h\in\mathcal{A}(X)$) form a closed basis, Cantor's Intersection Theorem implies that $f\perp g$ is equivalent to the following statement:
	\begin{center}
		``For all $h_1,\ldots,h_n\in\mathcal{A}(X)$, if $\bigcap_{i=1}^n[h_i=\theta_X]\cap[f=\theta_X]$ and $\bigcap_{i=1}^n[h_i=\theta_X]\cap[g=\theta_X]A$ are both empty, then $\bigcap_{i=1}^n[h_i=\theta_X]=\varnothing$.''
	\end{center}
	This condition is preserved under non-vanishing bijections, and so $T$ is a $\perp$-isomorphism.
	
	that $f\perpp g$ is equivalent to the following statement: ``There are finite families $\left\{a_i\right\}$, $\left\{b_j\right\}$ and $\left\{c_k\right\}$ in $\mathcal{A}(X)$ such that
	\begin{enumerate}[label=(\roman*)]
		\item\label{propositionnonvanishingbijectionisperppisomorphismproofitem1} $\bigcap_{i,j,k}[a_i=\theta_X]\cap[b_j=\theta_X]\cap[c_k=\theta_X]=\varnothing$;
		\item\label{propositionnonvanishingbijectionisperppisomorphismproofitem2} $a_i\perp b_j$ for all $i$ and $j$;
		\item\label{propositionnonvanishingbijectionisperppisomorphismproofitem3} $f\perp b_j$, $f\perp c_k$, $g\perp a_i$, and $g\perp c_k$ for all $i$, $j$ and $k$.''
	\end{enumerate}
	
	These statements should be interpreted as follows: Let $A=\bigcup_i[a_i\neq\theta_X]$, $B=\bigcup_j[b_j\neq\theta_X]$ and $C=\bigcup_k[c_k\neq\theta_X]$. Item \ref{propositionnonvanishingbijectionisperppisomorphismproofitem1} states that $A\cup B\cup C=X$, item \ref{propositionnonvanishingbijectionisperppisomorphismproofitem2} states that $A\cap B=\varnothing$, and item \ref{propositionnonvanishingbijectionisperppisomorphismproofitem3} states that $\supp(f)\cap B=\supp(f)\cap C=\supp(g)\cap A=\supp(g)\cap C=\varnothing$.
	In other words, $A$ and $B$ are two disjoint open sets used to separate $\supp(f)$ and $\supp(g)$, respectively, and the set $C$ covers the remainder of $X$ (while not intersecting $\supp(f)$ nor $\supp(g)$). The existence of such $A,B,C$ clearly implies $f\perpp g$, and the converse follows from regularity of $\mathcal{A}(X)$ and compactness arguments.
	
	Similar statements hold with $Y$ in place of $X$, and all of these properties are preserved by non-vanishing bijections. Therefore $T$ is a $\perpp$-isomorphism.
\end{proof}

\begin{theorem}\label{theoremnonvanishing}
	For every non-vanishing bijection $T\colon\mathcal{A}(X)\to\mathcal{A}(Y)$ there is a unique homeomorphism $\phi\colon Y\to X$ such that $[f=\theta_X]=\phi([Tf=\theta_Y])$ for all $f\in\mathcal{A}(X)$.
\end{theorem}
\begin{proof}
	By Proposition \ref{propositionnonvanishingbijectionisperppisomorphism}, we already know that $T$ is a $\perpp$-isomorphism, so let $\phi$ be the $T$-homeomorphism. Recall (Definition \ref{definitionsigma}) that $\sigma^{\theta_X}(f)=\operatorname{int}(\overline{[f\neq\theta_X]})$ and $Z^{\theta_X}(f)=\operatorname{int}([f=\theta_X])$ for all $f\in\mathcal{A}(X)$. Let us prove that
	\begin{align*}
	f(x)=\theta_X(x)\iff&\forall h_1\ldots,h_n\in\mathcal{A}(X),\\
	&\text{if }x\not\in\bigcup_{i=1}^n\sigma^{\theta_X}(h_i)\text{ then }\bigcap_{i=1}^n[h_i=\theta_X]\cap[f=\theta_X]\neq\varnothing,\ntag\label{equationnonvanishingproperty}
	\end{align*}
	Intuitively, the functions $h_i$ above should be thought of in such a way that $\bigcup_{i=1}n[h_i\neq\theta_X]$ is a ``large subset'' of $[f=\theta_X]\setminus\left\{x\right\}$.
	
	Formally, for the direction ``$\Rightarrow$'', assume that $f(x)=\theta_X(x)$ and $h_1,\ldots,h_n$ are such that $x\not\in\bigcup_i\sigma^{\theta_X}(h_i)$. Then $x\in\bigcap_i[h_i=\theta_X]\cap[f=\theta_X]$, and this set is nonempty.
	
	For the converse we prove the contrapositive. Assume that $f(x)\neq\theta_X(x)$. Regularity of $\mathcal{A}(X)$ and compactness of $[f=\theta_X]$ give us $h_1,\ldots,h_n\in\mathcal{A}(X)$ such that $h_i(x)\neq\theta_X(x)$ for all $i$, and $[f=\theta_X]\subseteq\bigcup_i[h_i\neq\theta_X]$, which negates the right-hand side of \eqref{equationnonvanishingproperty}.
	
	Since $\phi(\sigma^{\theta_Y}(Th))=\sigma^{\theta_X}(h)$ for all $h\in\mathcal{A}(X)$ and $T$ is non-vanishing, the condition of  \eqref{equationnonvanishingproperty} is preserved by $T$ and therefore, $\phi$ has the desired property.
\end{proof}

\section{Consequences}\label{sectionconsequences}

In this section we will recover several known results, some in greater generality than in their original statements, dealing with different algebraic structures on classes of continuous functions and their isomorphisms.

The general procedure we will use is the following: Suppose that $X$ and $Y$ are locally compact Hausdorff spaces, $H_X$ and $H_Y$ are Hausdorff spaces, $\theta_X\in C(X,H_X)$, $\theta_Y\in C(Y,H_Y)$ and $\mathcal{A}(X)$ and $\mathcal{A}(Y)$ regular subsets of $C_c(X,\theta_X)$ and $C_c(Y,\theta_Y)$, respectively.
\begin{enumerate}[label=\arabic*.]\label{generalprocedureforclassification}
	\item\label{generalprocedure1} Describe the relation $\perpp$ with the algebraic structure at hand. In general, we will first describe $\perp$ and use it in order to describe $\perpp$.
	\item\label{generalprocedure2} Given an algebraic isomorphism $T\colon\mathcal{A}(X)\to\mathcal{A}(Y)$ for appropriate classes of functions $\mathcal{A}(X)$ and $\mathcal{A}(Y)$, the first item ensures that $T$ is a $\perpp$-isomorphism, so let $\phi\colon Y\to X$ be the $T$-homeomorphism.
\end{enumerate}
At this point, we have proved that the algebraic isomorphism $T\colon\mathcal{A}(X)\to\mathcal{A}(Y)$ determines an isomorphism $\phi\colon Y\to X$. Since we moreover would like to describe $T$ in terms of $\phi$, we proceed as follows using the theory of Section \ref{sectionbasicmaps}:
\begin{enumerate}[label=\arabic*.]\setcounter{enumi}{2}
	\item\label{generalprocedure3} Prove that $T$ is $\phi$-basic, as in Definition \ref{definitionbasic}. Let $\chi$ be the $(\phi,T)$-transform.
	\item\label{generalprocedure4} Items \ref{generalprocedure1}-\ref{generalprocedure3} also apply to $T^{-1}$, which is thus also a basic $\perpp$-isomorphism.
	\item\label{generalprocedure5} Items \ref{generalprocedure3} and \ref{generalprocedure4}  imply, by Proposition \ref{propositionbasicisomorphisms}\ref{propositionbasicisomorphisms(a)} and \ref{propositionbasicisomorphisms(b)}, that the sections $\chi(y,\cdot)$ are injective and in the case that $\mathcal{A}(Y)$ is regular, item \ref{propositionbasicisomorphisms(d)} implies that $\chi(y,\cdot)$ is also surjective.
	\item\label{generalprocedure6} If we have a classification of algebraic isomorphisms between $H_X$ and $H_Y$, this classification will apply to each section $\chi(y,\cdot)$ of the $(\phi,T)$-transform, by Proposition \ref{propositionmodelmorphism}. This will in turn describe $T$ completely.
\end{enumerate}

\subsection{Milgram's Theorem}

In this subsection, we will always assume that $X$ (and similarly $Y$) is a locally compact Hausdorff space. We first generalize Milgram's theorem (and by consequence, Gelfand--Kolmogorov and Gelfand--Naimark) as follows: Let $S_X$ be a non-trivial path-connected Hausdorff topological monoid with zero $0$ and unit $1$, and which is both:
\begin{itemize}
	\item \emph{$0$-right cancellative}: for all $s,r,t\in S$, $st=rt\neq 0$ implies $s=r$;
	\item \emph{categorical at $0$}: $st=0$ implies $s=0$ or $t=0$.
\end{itemize}

Common examples of such semigroups are the following, under usual product: $\mathbb{R}$, $\mathbb{C}$, $[-1,1]$, the closed complex unit disc, $\operatorname{Gl}(n,\mathbb{R})\cup\left\{0\right\}$, and other variations.

We consider $\theta_X=0$, the zero map from $X$ to $S_X$. Let us also write $C_c(X,S_X)$ for $C_c(X,0)$ (to make the counter-domain explicit). Consider $Y$, $\theta_Y$ and $S_Y$ similarly. Urysohn's Lemma and path-connectedness of $S_X$ implies that $C_c(X,S_X)$ is regular.

Following the procedure in the beginning of this section, we first describe $\perpp$ in multiplicative terms:

\begin{lemma}\label{lemmamilgram}
	If $f,g\in C_c(X,S_X)$, then
	\[f\perpp g\iff \exists h\in C_c(X,S_X)\text{ such that }hf=f\text{ and }hg=0.\ntag\label{conditionmilgram}\]
\end{lemma}
\begin{proof}
	The condition in the right-hand side of \eqref{conditionmilgram} states that $h=1$ on $\supp(f)$ and $h\perp g$, so the implication ``$\Leftarrow$'' is immediate. The implication ``$\Rightarrow$'' (i.e., the existence of such $h$ provided $f\perpp g$) follows from Urysohn's Lemma and path-connectedness of $S_X$.\qedhere
\end{proof}

As a consequence, any multiplicative isomorphism $C_c(X,S_X)\to C_c(Y,S_Y)$ is a $\perpp$-isomorphism.

\begin{corollary}
	If $T\colon C(X,S_X)\to C_c(Y,S_Y)$ is a multiplicative isomorphism, then $X$ and $Y$ are homeomorphic.
\end{corollary}

To recover Milgram's original theorem in full generality, we restrict now to the case $S_X=S_Y=\mathbb{R}$ in order to obtain an explicit description of $T$ as above.

First, we will need to recall the classification of multiplicative isomorphisms of $\mathbb{R}$. The general case may be found in \cite[Theorem 3.1.3]{MR2467621}, but for our purposes it will suffice to consider only the continuous isomorphisms (which were already described in \cite[Lemma 4.3]{MR0029476}, for example).

\begin{proposition}\label{theoremclassificationmultiplicativeisomorphismsofr}
	Let $\tau\colon\mathbb{R}\to\mathbb{R}$ be a multiplicative isomorphism. Then $\tau$ is continuous if and only if $0<x<1$ implies $0<\tau(x)<1$. In this case, $\tau$ has the form $\tau(x)=\operatorname{sgn}(x)|x|^p$ for some $p>0$.
\end{proposition}

We will now classify multiplicative isomorphisms from $C_c(X,\mathbb{R})$ and $C_c(Y,\mathbb{R})$. Note that the following theorem is a generalization of \cite[Theorem A]{MR0029476} to the \emph{locally} compact setting.

\begin{theorem}[{Milgram's Theorem \cite[Theorem A]{MR0029476}, for locally compact spaces}]\label{theoremmilgram}
	Let $X$ and $Y$ be locally compact Hausdorff spaces and let $T\colon C_c(X,\mathbb{R})\to C_c(Y,\mathbb{R})$ be a multiplicative isomorphism. Then there exists a homeomorphism $\phi\colon Y\to X$, a closed, discrete and isolated subset $F\subseteq Y$, and a continuous positive function $p\colon Y\setminus F\to(0,\infty)$ satisfying
	\[Tf(y)=\operatorname{sgn}(f(\phi(y)))|f(\phi (y))|^{p(y)}\]
	for all $f\in C_c(X,\mathbb{R})$ and $y\in Y\setminus F$.
\end{theorem}
\begin{proof}
	By Lemma \ref{lemmamilgram}, $T$ is a $\perpp$-isomorphism, so let $\phi$ be the $T$-homeomorphism. We prove that $T$ is $\phi$-basic in a few simple steps.
	\begin{enumerate}[label=\arabic*.]
		\item\label{theoremmilgramstep1} \uline{If $f,h\in C_c(X,\mathbb{R})$, then $f=1$ on $\supp(h)$ (i.e., $fh=h$) if and only if $Tf=1$ on $\supp(Th)$ (i.e., $(Tf)(Th)=(Th)$).} This implies that
		
		\item\label{theoremmilgramstep2} \uline{If $f,h\in C_c(X,\mathbb{R})$, then $f\neq 0$ on $\supp(h)$ if and only if $Tf\neq 0$ on $\supp(Th)$.}
		
		Indeed, if $f\neq 0$ on $\supp(h)$, we can find $g\in C_c(X,\mathbb{R})$ such that $g=1/f$ on $\supp(h)$. Item \ref{theoremmilgramstep1} implies that $TfTg=1$ on $\supp(Th)$, and in particular $Tf\neq 0$ on $\supp(Th)$. As the sets of the form $\sigma(h)$ form a basis of $X$, we conclude that
		
		\item\label{theoremmilgramstep3} \uline{If $f\in C_c(X,\mathbb{R})$ and $y\in Y$, then $f(\phi(y))\neq 0$ if and only if $Tf(y)\neq 0$.}
		
		A similar argument to that of item \ref{theoremmilgramstep2} also proves that
		
		\item\label{theoremmilgramstep4} \uline{If $f,g,h\in C_c(X,\mathbb{R})$, then $f$ and $g$ coincide and are nonzero on $\supp(h)$ if and only if $Tf$ and $Tg$ coincide and are nonzero on $\supp(h)$.}
		
		\item\label{theoremmilgramstep5} In particular, from \ref{theoremmilgramstep3}, \uline{$f(\phi(y))=0$ if and only if $Tf(y)=0$.}
		\item\label{theoremmilgramstep6} \uline{If $f\in C_c(X,\mathbb{R})$ and $y\in Y$, then $f(\phi(y))=1$ if and only if $Tf(y)=1$:}
		
		Suppose this was not the case, say $f(\phi(y))=1$ but $Tf(y)\neq 1$. From item \ref{theoremmilgramstep1} above we know that $y$ is not isolated. If necessary, as $T$ is a $\perp$-isomorphism, Theorem \ref{theoremdisjoint} allows us to change $f$ by a function which equals $1/f$ on some neighbourhood of $\phi(y)$ and assume that $Tf(y)>1$. The rest of the argument is similar to that of \cite[Lemma 4.1]{MR0029476}: We first take a sequence of distinct points $y_n$, all contained in some compact neighbourhood of $y$, satisfying
		\begin{itemize}
			\item $f(\phi(y_n))^n\to 1=f(\phi(y))$; and 
			\item $Tf(y_n)\to Tf(y)$, so in particular $Tf(y_n)^n\to\infty$.
		\end{itemize}
		Using Propositions \ref{propositiondisjointopensets} and \ref{propositionconstructfunction}\ref{propositionconstructfunction(b)}, we may find $g\in C_c(X,\mathbb{R})$ which coincides with $f^n$ on some neighbourhood of $\phi(y_n)$. But then $Tg$ coincides with $(Tf)^n$ on a neighbourhood $y_n$ (because $T$ is a $\perp$-isomorphism), contradicting the fact that $Tg$ is bounded.
	\end{enumerate}
	
	We conclude, from \ref{theoremmilgramstep5} and \ref{theoremmilgramstep6}, that $T$ is $\phi$-basic. Let $\chi\colon Y\times\mathbb{R}\to\mathbb{R}$ be the $T$-transform.
	
	Let $F=\left\{y\in Y:\chi(y,\cdot)\text{ is discontinuous}\right\}$. Let us prove that $F$ is closed and discrete, or equivalently that $F\cap K$ is finite for all compact $K\subseteq Y$. Otherwise, by Proposition \ref{theoremclassificationmultiplicativeisomorphismsofr}, there would be distinct $y_1,y_2,\ldots\in K$ and a strictly decreasing sequence $t_n\to 0$ such that $\chi(y_n,t_n)>n$. Going to a subsequence if necessary, we can construct, by Proposition \ref{propositionconstructfunction}\ref{propositionconstructfunction(b)}, $f\in C_c(X,\mathbb{R})$ with $f(\phi(y_n))=t_n$, so $Tf(y_n)>n$ for all $n$, a contradiction to the boundedness of $Tf$.
	
	To prove that $F$ is isolated we prove that it is open: If $z\in F$, then there is $t\in(0,1)$ with $\chi(z,t)>1$. Take $f\in C_c(X,\mathbb{R})$ such that $f=t$ on a neighbourhood $U$ of $\phi(z)$. In particular $Tf(z)>1$, so there is a neighbourhood $W$ of $z$ such that $Tf>1$ on $W$. Then for all $y\in\phi^{-1}(U)\cap W$, $\chi(y,t)=Tf(y)>1$, so $y\in F$.
	
	Therefore, $F$ consists of isolated points, since $Y$ is locally compact. For $y\not\in F$, Proposition \ref{propositionmodelmorphism} and Proposition \ref{theoremclassificationmultiplicativeisomorphismsofr} imply that $\chi(y,\cdot)$ has the form $\chi(y,t)=\operatorname{sgn}(t)|t|^{p(y)}$ for some $p(y)>0$. As for the continuity of $p$, Let $U$ be any relatively compact open subset of $Y$ not intersecting $F$, and $f\in C_c(X,\mathbb{R})$ with $f=2$ on $\phi(U)$. Then $Tf(y)=f(y)^{p(y)}=2^p(y)$, so $p=\log_{2}\circ Tf$ on $U$. This proves that $T$ is continuous on $Y\setminus F$.\qedhere
\end{proof}

\subsection{Group-Valued functions; the Hernández--Ródenas Theorem}

Given topological groups $G$ and $H$, denote by $\operatorname{AbsIso}(G,H)$ the set of algebraic group isomorphisms from $G$ to $H$, and by $\operatorname{TopIso}(G,H)$ the set of topological (i.e., homeomorphisms which are) group isomorphisms.

Let $X$ and $Y$ be compact Hausdorff spaces and $G$ a Hausdorff topological group. In \cite[Theorem 3.7]{MR2324919}, Hernández and Ródenas classified non-vanishing group morphisms (not necessarily \emph{isomorphisms}) $T\colon C(X,G)\to C(Y,G)$ which satisfy the following properties:
\begin{enumerate}[label=(\roman*)]
	\item There exists a continuous group morphism $\psi\colon G\to C(X,G)$, where $C(X,G)$ is endowed with the topology of pointwise convergence, such that for all $\alpha\in G$ and all $y\in Y$, $T(\psi(\alpha))(y)=\alpha$;
	\item For every continuous endomorphism $\theta\colon G\to G$ and every $f\in C(X,G)$, $T(\theta\circ f)=\theta\circ (Tf)$.
\end{enumerate}
If $T$ is a group isomorphism and $T^{-1}$ is continuous (with respect to uniform convergence) then condition (i) is immediately satisfied, however this is not true for (ii): For example, if $\operatorname{TopIso}(G,G)$ is non-abelian and $\rho\in\operatorname{TopIso}(G,G)$ is any non-central element, then the map $T\colon C(X,G)\to C(X,G)$ given by $Tf=\rho\circ f$ is a group isomorphism, and a self-homeomorphism of $C(X,G)$ with the topology of uniform convergence, which does not satisfy (ii).

In the next theorem we obtain the same type of classification as in \cite[Theorem 3.7]{MR2324919}, without assuming condition (ii), however we consider only non-vanishing group isomorphisms. Given a compact Hausdorff space $X$, a topological group $G$ and $\alpha\in G$, denote by $\overline{\alpha}$ the constant function $X\to G$, $x\mapsto \alpha$. We endow $C(X,G)$ with the topology of pointwise convergence.

\begin{theorem}
	Suppose that $G$ and $H$ are Hausdorff topological groups, and $X$ and $Y$ are compact Hausdorff spaces for which $(X,\overline{1_G},C(X,G))$ and $(Y,\overline{1_H},C(Y,H))$ are regular.
	
	Let $T\colon C(X,G)\to C(Y,H)$ be a non-vanishing group isomorphism (Definition \ref{definitionnonvanishingbijection}). Then there exist a homeomorphism $\phi\colon Y\to X$ and a map $w\colon Y\to\operatorname{AbsIso}(H,G)$ such that $Tf(y)=w(y)(f(\phi(y))$ for all $y\in Y$ and $f\in C(X,G)$.
	
	If $T$ is continuous on the constant functions then each $w(y)$ is continuous and $T$ is continuous on $C(X,G)$. If both $T$ and $T^{-1}$ are continuous on the constant functions then $w(y)\in\operatorname{TopIso}(H,G)$ and $T$ is a homeomorphism for the topologies of pointwise convergence.
\end{theorem}
\begin{proof}
	By Theorem \ref{theoremnonvanishing} and Proposition \ref{propositionbasicnessofgroupvalued}, there is a homeomorphism $\phi\colon Y\to X$ such that $T$ is $\phi$-basic, and the sections $\chi(y,\cdot)\colon G\to H$ of the $(\phi,T)$-transform $\chi$ are group morphisms by Proposition \ref{propositionmodelmorphism}. Similar facts hold for $T^{-1}$, so Proposition \ref{propositionbasicisomorphisms} implies that each section $\chi(y,\cdot)$ is bijective. Letting $w(y)=\chi(y,\cdot)$ we are done with the first part.
	
	Now note that for all $\alpha\in G$ and $y\in Y$,
	\[w(y)(\alpha)=w(y)(\overline{\alpha}(\phi(y)))=T\overline{\alpha}(y)\]
	which implies that every $w(y)$ is continuous if and only if $T$ is continuous on the constant functions (because the map $\alpha\mapsto\overline{\alpha}$ from $G$ to $C(X,G)$ is a homeomorphism onto its image). In this case, from the equality
	\[Tf(y)=w(y)(f(\phi(y))\qquad\text{for all}\quad f\in C(X,G)\quad\text{and}\quad y\in Y,\]
	we can readily see that $T$ is continuous. The last part, assuming also that $T^{-1}$ is continuous on the constant functions, is similar, using $T^{-1}$ and $w(y)^{-1}$ in place of $T$ and $w(y)$.
\end{proof}

\subsection{Kaplansky's Theorem}\label{subsectionkaplansky}

Let $R$ be a totally ordered set without supremum nor infimum, regarded as a topological space with the order topology, and let $X$ be a locally compact Hausdorff space. We consider the pointwise order on $C(X,R)$: $f\leq g$ if and only if $f(x)\leq g(x)$ for all $x$, which makes $C(X,R)$ a lattice: for all $f,g\in C(X,R)$ and $x\in X$,
\[(f\lor g)(x)=\max\left\{f(x),g(x)\right\},\qquad\text{and}\qquad(f\land g)(x)=\min\left\{f(x),g(x)\right\}.\]
Denote by $C_b(X,R)$ the sublattice of bounded continuous functions from $X$ to $R$.

In \cite{MR0020715} Kaplansky proved that if $X$ is compact and \emph{$R$-normal} (this is a stronger version of strong regularity; see \cite[p.\ 618]{MR0020715}), then the lattice $C(X,R)$ determines $X$ completely, and in \cite{MR0026240}, classified additive lattice isomorphisms between these lattices of functions in the case that $R=\mathbb{R}$. We improve on these results in the following ways: We allow $X$ to be non-compact (only locally compact), obtain a recovery theorem for $X$ from a subcollection $\mathcal{A}$ of $C_b(X,R)$ (Theorem \ref{theoremkaplanskywithoutlowerbound}), and classify lattice isomorphisms in the case of non-real-valued functions for first-countable spaces (Theorem \ref{theoremclassificationlatticemorphismfirstcountable}).

We will consider sublattices $\mathcal{A}$ of $C_b(X,R)$ which satisfy
\begin{enumerate}[label=(L\arabic*)]
	\item\label{conditionskaplansky1} for all $f,g\in\mathcal{A}$, $\overline{[f\neq g]}$ is compact;
	\item\label{conditionskaplansky2} for all $f\in\mathcal{A}$, every open set $U\subseteq X$, every $x\in U$ and $\alpha\in R$, there exists $g\in\mathcal{A}$ such that $g(x)=\alpha$ and $[g\neq f]\subseteq U$.
\end{enumerate}

Condition \ref{conditionskaplansky1} is equivalent to saying that $\mathcal{A}\subseteq C_c(X,\theta)$, where $\theta$ is some (any) element of $\mathcal{A}$, and condition \ref{conditionskaplansky2} is a form of regularity. These conditions are satisfied in the settings that have been previously studied.

\begin{example}[{Kaplansky, \cite{MR0020715}}]\label{examplekaplanskyisl1l2}
	Suppose that $X$ is compact and $\mathcal{A}$ is an $R$-normal sublattice of $C(X,R)$. Condition \ref{conditionskaplansky1} is trivial, so let us check that $\mathcal{A}$ satisfies \ref{conditionskaplansky2}: Suppose $f$, $U$, $x$ and $\alpha$ are as in that condition. For the sake of the argument we can assume $f(x)\leq\alpha$. Let $\beta$ be any lower bound of $f(X)$, and from $R$-normality find $h\in\mathcal{A}$ such that $h(x)=\alpha$ and $h=\beta$ outside $U$. Then $g=f\lor h$ has the desired properties.
\end{example}

\begin{example}[{Li--Wong,  \cite{MR3162258}}]
	Suppose that $X$ is compact, $R=\mathbb{R}$, and $\mathcal{A}$ is a regular additive subgroup of $C(X,\mathbb{R})$, so \ref{conditionskaplansky1} is also trivial and again we need to verify condition \ref{conditionskaplansky2}: Let $f$, $U$, $x$ and $\alpha$ be as in \ref{conditionskaplansky2}. By regularity, take $h$ such that $\supp(h)\subseteq U$ and $h(x)=\alpha-f(x)$. Then $g=f+h$ has the desired properties
\end{example}

We will now recover the main result of \cite{MR0020715}, The following lemma is based on \cite{MR0052760}.

\begin{lemma}\label{importantlemmakaplansky}
	Let $\mathcal{A}$ be a sublattice of $C_b(X,R)$ satisfying \ref{conditionskaplansky1} and \ref{conditionskaplansky2}, and let $f_0$ be any element of $\mathcal{A}$. Let $\mathcal{A}_{\geq f_0}=\left\{f\in\mathcal{A}:f\geq f_0\right\}$. Then $(X,f_0,\mathcal{A}_{\geq f_0})$ is weakly regular and for $f,g\in\mathcal{A}_{\geq f_0}$,
	\begin{enumerate}[label=(\alph*)]
		\item\label{importantlemmakaplansky(a)} $f\perp g\iff f\land g=f_0$ (which is the minimum of $\mathcal{A}_{\geq f_0})$;
		\item\label{importantlemmakaplansky(b)} $f\Subset g$ is equivalent to the following statement:
		\[\begin{split}
		\text{``for every bounded subset $\H\subseteq\mathcal{A}$ such that $h\subseteq f$ for all $h\in \H$,}\\
		\text{there is an upper bound $k$ of $\H$ such that $k\subseteq g$.''}\end{split}\tag{K}\]
	\end{enumerate}
\end{lemma}

\begin{proof}
	Weak regularity is immediate from \ref{conditionskaplansky2} and the fact that $R$ does not have a supremum, and item \ref{importantlemmakaplansky(a)} is trivial. Let us prove \ref{importantlemmakaplansky(b)}. First suppose $f\Subset g$ and $\H\subseteq\mathcal{A}$ is a bounded subset such that $h\subseteq f$ for all $h\in \H$. Let $\alpha\in R$ be an upper bound of $\bigcup_{h\in\H}h(X)$. From weak regularity and compactness of $\supp(f)$, we can take finitely many functions $k_1,\ldots,k_n$ such that $k_i\Subset g$, and for every $x\in\supp(f)$ there is some $i$ with $k_i(x)>\alpha$. Letting $k=\bigvee_{i=1}^n k_i$ we obtain the desired properties.
	
	Conversely, suppose that condition (K) holds. Let $\alpha$ be any upper bound of $f_0(X)$ and again take $\beta>\alpha$. Let $\mathscr{H}=\left\{h\in\mathcal{A}_{\geq f_0}:h\leq \beta, h\subseteq f\right\}$. Let $k$ be an upper bound of $\H$ with $k\subseteq g$. By Property \ref{conditionskaplansky2}, we have $\sigma(f)=\bigcup_{h\in\mathscr{H}}\sigma(h)$, so $k\geq\beta$ on $\sigma(f)$ and thus also on $\supp(f)$, which implies $f\Subset k$. Since $k\subseteq g$ then $f\Subset g$.
\end{proof}

As an immediate consequence of Lemma \ref{importantlemmakaplansky} and Theorem \ref{maintheorem} we have the following generalization of Kaplansky's Theorem:

\begin{theorem}[Kaplansky \cite{MR0020715}]\label{theoremkaplansky}
	Suppose $R$ has no supremum nor infimum, $\mathcal{A}(X)$ and $\mathcal{A}(Y)$ are sublattices of $C_b(X,R)$ and $C_b(Y,R)$, respectively, satisfying conditions \ref{conditionskaplansky1} and \ref{conditionskaplansky2}, and $T\colon\mathcal{A}(X)\to\mathcal{A}(Y)$ is a lattice isomorphism. Then for any $f_0\in\mathcal{A}$, $T$ restricts to a $\perpp$-isomorphism of the regular sublattices $\mathcal{A}(X)_{\geq f_0}$ and $\mathcal{A}(Y)_{\geq Tf_0}$. In particular, $X$ and $Y$ are homeomorphic.
\end{theorem}

Our immediate goal is to prove that the homeomorphism between $X$ and $Y$ given by Theorem \ref{theoremkaplansky} does not depend on the choice of function $f_0$ in Lemma \ref{importantlemmakaplansky}.

\begin{lemma}\label{lemmarestrictionoflatticeisomorphismisbasic}\index{$T$-homeomorphism}
	Under the conditions of Theorem \ref{theoremkaplansky}, let $\phi\colon Y\to X$ be the $T|_{\mathcal{A}(X)_{\geq f_0}}$-homeomorphism. If $f,g\in\mathcal{A}(X)_{\geq{f_0}}$ then $\phi^{-1}(\operatorname{int}([f=g]))=\operatorname{int}([Tf=Tg])$.
\end{lemma}
\begin{proof}
	We will use the superscript ``$f_0$'' as in Definition \ref{definitionsigma}, that is, for all $f\geq f_0$.
	\[\sigma^{f_0}(f)=\operatorname{int}(\overline{[f\neq f_0]}),\qquad\text{and}\qquad Z^{f_0}(f)=\operatorname{int}([f=f_0]).\]
	and similarly with $Tf_0$ in place of $f_0$. Then $\phi$ is the only homeomorphism satisfying $\sigma^{f_0}(f)=\phi(\sigma^{Tf_0}(Tf))$, or equivalently $Z^{f_0}(f)=\phi(Z^{Tf_0}(Tf))$, for all $f\geq f_0$.
	
	First, assume that $f\leq g$, and let $U=\operatorname{int}([f=g])$. For all $x\in[f<g]$, choose a function $k_x\in\mathcal{A}(X)_{\geq f_0}$ such that
	\begin{itemize}
		\item $\sigma^{f_0}(k_x)\cap [f=g]=\varnothing$;
		\item $k_x(x)=g(x)$;
		\item $k_x\leq g$.
	\end{itemize}
	Then $g=\sup\left\{f,k_x:x\in[f<g]\right\}$, so $Tg=\sup\left\{Tf,Tk_x:x\in[f<g]\right\}$..
	
	Let us prove that $Tf(y)=Tg(y)$ for all $y\in\phi(U)$. Since $U\subseteq\bigcap_{x\in[f<g]}Z^{f_0}(k_x)$, then $\phi^{-1}(U)\subseteq\bigcap_{x\in[x<g]}Z^{Tf_0}(Tk_x)$.
	
	Given $y\in\phi(U)$, use property \ref{conditionskaplansky2} to find $h\in\mathcal{A}(Y)$ such that $h(y)=Tf(y)$ and $h=Tg$ outside of $\phi(U)$. Then $h'=(Tf\lor h)\land Tg$ is an upper bound of $\left\{Tf,Tk_x:x\in[f<g]\right\}$, so it is also an upper bound of $Tg$, and in particular
	\begin{align*}
	Tg(y)\leq ((Tf\lor h)\land Tg)(y)=Tf(y)\leq (Tf\lor h)(y)=Tf(y),
	\end{align*}
	so $Tg(y)=Tf(y)$.
	
	In the general case, if $f=g$ on an open set $U$ then $f=f\land g=g$ on $U$, so the previous case implies that $Tf$ and $Tg$ both coincide with $T(f\land g)$ on $\phi^{-1}(U)$.
\end{proof}

\begin{theorem}\label{theoremkaplanskywithoutlowerbound}
	Under the conditions of Theorem \ref{theoremkaplansky}, there exists a unique homeomorphism $\phi\colon Y\to X$ such that $\phi(\operatorname{int}([Tf=Tg]))=\operatorname{int}([f=g])$ for all $f,g\in\mathcal{A}(X)$. (In this case we will still call $\phi$ the \emph{$T$-homeomorphism}.)
\end{theorem}
\begin{proof}
	For each $f_0\in \mathcal{A}(X)$, let $\phi^{f_0}\colon Y\to X$ be the $T|_{\mathcal{A}(X)_{\geq f_0}}$-homeomorphism. Given $f_0,g_0\in\mathcal{A}(X)$, Lemma \ref{lemmarestrictionoflatticeisomorphismisbasic} implies that $\phi^{f_0\land g_0}$ satisfies the property of both the $T|_{\mathcal{A}(X)_{\geq f_0}}$ and the $T|_{\mathcal{A}(X)_{\geq g_0}}$-homeomorphisms, so $\phi^{f_0}=\phi^{f_0\land g_0}=\phi^{g_0}$. We are done by letting $\phi=\phi^{f_0}$ for some arbitrary $f_0\in\mathcal{A}(X)$.
\end{proof}

A natural goal now is to classify the lattice isomorphisms as given in Theorem \ref{theoremkaplanskywithoutlowerbound}, which is possible when we consider first-countable spaces. A similar argument to that of Theorem \ref{propositionbasicisomorphisms} appears in \cite{MR2995073}, although in a different context (considering lattices of possibly unbounded real-valued continuous functions on complete metric spaces). See also \cite{MR2999998,MR3404615}.

Let us reinforce the assumptions, assumed throughout this subsection, that $X$ denotes a locally compact Hausdorff space and $R$ is a totally ordered set without maximum or minimum with the order topology. The following is a version of Proposition \ref{propositionconstructfunction} in this setting, and is also proven by ``cutting and pasting'' (aided by the lattice operations of $R$), so we omit the details.

\begin{proposition}\label{propositionconstructfunctionkaplansky}
	Assume further that $X$ and $R$ are first-countable, and that $\theta\in C_b(X,R)$ is such that $C_c(X,\theta)$ satisfies \ref{conditionskaplansky2}. Suppose that $\left\{x_n\right\}_n$ is an injective sequence in $X$, converging to $x_\infty\in X$. Let $g_n\in C(X,R)$ be functions such that $g_n(x_n)\to \theta(x)$.
	
	Then there exists $f\in C_c(X,\theta)$ such that $f=g_n$ on a neighbourhood of $x_n$ for each $n$, and that $f=\theta$ outside of a compact containing $\left\{x_n:n\in\mathbb{N}\right\}$ (which may be taken as small as desired with this property).
\end{proposition}

\begin{theorem}\label{theoremclassificationlatticemorphismfirstcountable}
	Suppose that $X$, $Y$ and $R$ are first-countable, that $C_c(X,\theta_X)$ and $C_c(Y,\theta_Y)$ satisfy \ref{conditionskaplansky2}, and that $T\colon C_c(X,\theta_X)\to C_c(Y,\theta_Y)$ is a lattice isomorphism. Then there are a unique homeomorphism $\phi\colon Y\to X$ and a continuous function $\chi\colon Y\times R\to R$ such that
	\[Tf(y)=\chi(y,f(\phi(y)))\qquad\text{ for all }y\in Y\text{ and }f\in C_c(X,0)\ntag\label{equationclassificationoforderkaplanskyisomorphismsforfirstcountble}\]
	and $\chi(y,\cdot)\colon R\to R$ is an increasing bijection for each $y\in Y$.
\end{theorem}
\begin{proof}
	Let $\phi\colon Y\to X$ be the $T$-homeomorphism. We just need to prove that $T$ is $\phi$-basic, so assume $y\in Y$ and $f(\phi(y))=g(\phi(y))$. In order to prove that $Tf(y)=Tg(y)$, we may assume that $f\leq g$, by considering the auxiliary function $f\land g$.
	
	If $y$ is isolated in $Y$, then $f$ and $g$ coincide on the open set $\left\{\phi(y)\right\}$, so $Tf$ and $Tg$ coincide on the open set $\left\{y\right\}$.
	
	Assume then that $y$ is not isolated. Since $Y$ is first-countable, let $(y_n)_n$ be an injective sequence in $Y$ converging to $y$. By Proposition \ref{propositionconstructfunctionkaplansky}, there is $h\in C_c(X,\theta)$ such that
	\begin{itemize}
		\item If $n$ is even, $h=f$ on a neighbourhood of $\phi(y_n)$;
		\item If $n$ is odd, $h=g$ on a neighbourhood of $\phi(y_n)$.
	\end{itemize}
	Then $\phi(y)\in\overline{\operatorname{int}[f=h]}$, so $t\in\overline{\operatorname{int}[Tf=Th]}$ and so $Tf(y)=Th(y)$. Similarly, $Tg(y)=Th(y)=Tf(y)$. This proves that $T$ is $\phi$-basic. Let $\chi$ be the $(\phi,T)$-transform.
	
	Proposition \ref{propositionmodelmorphism}, applied to the signature of lattices (with the binary symbol ``$\lor$'' interpreted as ``join'') implies that the sections $\chi(y,\cdot)$ are lattice isomorphisms of $R$ for all $n$, and in particular homeomorphisms. The proof that $\chi$ is continuous is similar to that of implication \ref{theorembasicisomorphisms(2)}$\Rightarrow$\ref{theorembasicisomorphisms(1)} of Theorem \ref{theorembasicisomorphisms} (using Proposition \ref{propositionconstructfunctionkaplansky} instead of \ref{propositionconstructfunction}).
\end{proof}

In the case of additive lattice isomorphisms of spaces of real-valued functions, we do not require the first-countability hypothesis.

\begin{theorem}\label{theoremadditivekaplansky}
	Suppose $R=\mathbb{R}$, and $T\colon C_c(X)\to C_c(Y)$ is an additive lattice isomorphism. Then there are a unique homeomorphism $\phi\colon Y\to X$ and a unique positive continuous function $p\colon Y\to (0,\infty)$ such that $Tf(y)=p(y)f(\phi(y))$ for all $f\in C_c(X)$ and $y\in Y$.
\end{theorem}
\begin{proof}
	
	First note that for all $f\in C_c(X)$, $|f|=(f\lor 0)-(f\land 0)$, so $T|f|=|Tf|$. Let $\phi\colon Y\to X$ be the $T$-homeomorphism, given by Theorem \ref{theoremkaplanskywithoutlowerbound}.
	
	The proof that $T$ is $\phi$-basic is similar to that of item \ref{theoremmilgramstep6} of the proof of Theorem \ref{theoremmilgram}. Let $\chi$ be the $T$-transform. Each section $\chi(y,\cdot)$ is an additive order-preserving bijection (Propositions \ref{propositionmodelmorphism} and \ref{propositionbasicisomorphisms}) and hence has the form $\chi(y,t)=p(y)t$ for some $p(y)>0$. If $Tf(y)\neq 0$, then $f(\phi(y))\neq 0$ as well and $p=Tf/(f\circ \phi)$ on a neighbourhood of $y$, thus $p$ is continuous.\qedhere
\end{proof}


\subsection{Li--Wong Theorem}\label{subsectionliwong}

We will now recover a theorem of Li and Wong, \cite[Theorem 2.2]{MR3162258}, which can be seen as a generalization of Theorem \ref{theoremadditivekaplansky}. We will proceed in the opposite direction, i.e., by proving their result instead as a consequence of the more general Theorem \ref{theoremkaplansky}. Let $\mathbb{K}=\mathbb{R}$ or $\mathbb{C}$.

\begin{theorem}[Li--Wong \cite{MR3162258}]\label{theoremliwong}
	Let $X$ and $Y$ be compact Hausdorff spaces, and $\mathcal{A}(X)$ and $\mathcal{A}(Y)$ be two regular vector sublattices of $C(X,\mathbb{K})$ and $C(Y,\mathbb{K})$, respectively. Suppose that $T\colon \mathcal{A}(X)\to \mathcal{A}(Y)$ is a $\mathbb{K}$-linear isomorphism which preserves non-vanishing functions, that is, for all $f\in\mathcal{A}(X)$,
	\[0\in f(X)\iff 0\in Tf(Y).\]
	Then there is a homeomorphism $\phi\colon Y\to X$ and a continuous non-vanishing function $p\colon Y\to\mathbb{K}$ such that $Tf(y)=p(y)f(\phi(y))$ for all $f\in\mathcal{A}(X)$ and $y\in Y$.
\end{theorem}

The following technical lemma is the main necessary tool of the proof. We do not assume that $\mathcal{A}(X)$ contains the constant functions, however since it is a regular lattice then it contains a strictly positive function $F$ satisfying $0<F<1/2$ (see the beginning of the proof of Theorem \ref{theoremliwong} below). The use of the constant function ``$1/2$'' in the proof of \cite[Lemma 2.3]{MR3162258} can be replaced by $F$.

\begin{lemma}[{\cite[Lemma 2.3]{MR3162258}}]
	Any $T$ as in Theorem \ref{theoremliwong} is a $\perp$-isomorphism.
\end{lemma}

\begin{proof}[Proof of Theorem \ref{theoremliwong}]
	We will use Theorem \ref{theoremkaplanskywithoutlowerbound}. First we need to modify $T$ to obtain a lattice isomorphism. Since $\mathcal{A}(X)$ is a sublattice, then for all $f\in\mathcal{A}(X)$,
	\[f^+=\max(f,0),\quad f^-=\max(-f,0)\quad\text{and}\quad|f|=f^++f^-\text{ belong to }\mathcal{A}(X).\]
	
	As $\mathcal{A}(X)$ is regular and $X$ is compact, we can take finitely many functions $f_1,\ldots,f_n\in\mathcal{A}(X)$ such that for all $x\in X$, $f_i(x)\neq 0$ for some $i$, and therefore the function $F=\sum_{i=1}^n|f_i|$ belongs to $\mathcal{A}(X)$ and is non-vanishing, so $TF$ is also non-vanishing. We define new classes of functions
	\[\mathcal{B}(X)=\left\{f/F:f\in\mathcal{A}(X)\right\},\qquad\mathcal{B}(Y)=\left\{G/TF:g\in\mathcal{A}(Y)\right\}.\]
	It is immediate to see that $\mathcal{B}(X)$ and $\mathcal{B}(Y)$ are regular, and contain the constant functions of $X$ and $Y$, respectively. Define a linear isomorphism $S\colon \mathcal{B}(X)\to\mathcal{B}(Y)$, $S(f)=T(fF)/TF$, which preserves non-vanishing functions and satisfies $S(1)=1$. Given a scalar $\lambda$, linearity and the non-vanishing property of $S$ imply that, for all $f\in\mathcal{B}(X)$,
	\begin{align*}
	\lambda\not\in f(X)&\iff f-\lambda\text{ is non-vanishing}\\
	&\iff Sf-\lambda\text{ is non-vanishing}\iff \lambda\not\in Sf(Y),
	\end{align*}
	so $f(X)=Sf(Y)$, i.e., $S$ preserves images of functions.
	
	However, in order to apply Theorem \ref{theoremkaplanskywithoutlowerbound} we also need to make sure that $\mathcal{B}(X)$ and $\mathcal{B}(Y)$ are lattices. As $F>0$, it readily follows that $\mathcal{B}(X)$ is a (self-adjoint) sublattice of $C(X,\mathbb{K})$, however this is not so immediate for $\mathcal{B}(Y)$.
	
	As $S$ preserves images of functions, it preserves real functions. If $f\in\mathcal{B}(X)$, then $S(\operatorname{Re}(f))$ and $S(\operatorname{Im}(f))$ are real functions such that $Sf=S(\operatorname{Re}(f))+iS(\operatorname{Im}(f))$. As $T$ is a $\perp$-isomorphism then $S$ is also a $\perp$-isomorphism, so we also obtain $S(\operatorname{Re}(f))\perp S(\operatorname{Im}(f))$. This is enough to conclude that $S$ preserves real and imaginary parts of functions, from which it follows that $\mathcal{B}(Y)$ is self-adjoint. Similarly, $S$ preserves positive and negative parts of functions. In particular, if $f\in\mathcal{B}(Y)$ then $f^+\in\mathcal{B}(Y)$, and this is enough to conclude that $\mathcal{B}(Y)$ is a sublattice of $C(Y,\mathbb{K})$, and that $S$ is an order-preserving isomorphism.
	
	We may then consider only real-valued functions, and the complex case will follow by linearity (and since $S$ preserves real and imaginary parts). By Kaplansky's Theorem (\ref{theoremkaplanskywithoutlowerbound}), we can construct the $S$-homeomorphism $\phi\colon Y\to X$. Now we need to prove that $S$ is $\phi$-basic.
	
	Suppose $f(x)\neq 0$ for a given $x\in X$, and let us assume, without loss of generality, that $f(x)>0$. Then $f>0$ on some neighbourhood $U$ of $x$. Again using compactness of $X\setminus U$ and regularity of the sublattice $\mathcal{B}(X)$ we can construct a function $g\in\mathcal{B}(X)$ such that $g=0$ on some neighbourhood of $x$ and $g>0$ on $X\setminus U$. Letting $\widetilde{f}=f\lor g$, we have $\widetilde{f}=f$ on some neighbourhood of $x$, so $S\widetilde{f}=Sf$ on some neighbourhood of $\phi^{-1}(x)$. But $\widetilde{f}$ is non-vanishing, so $S\widetilde{f}$ is also non vanishing and in particular $Sf(\phi^{-1}(x))\neq 0$. This proves that $S$ is basic with respect to $\phi$.
	
	Letting $\chi\colon Y\times\mathbb{R}\to\mathbb{R}$ be the $S$-transform, we have that all sections $\chi(y,\cdot)$ are linear and increasing (Theorem \ref{propositionmodelmorphism}), hence of the form $\chi(y,t)=P(y)t$ for a certain $P(y)>0$. Since $S(1)=1$ then $P=1$, that is, $\chi(y,t)=t$ for all $t\in\mathbb{R}$.
	
	Finally, for all $f\in\mathcal{A}(X)$ and $y\in Y$,
	\[Tf(y)=(TF)(y)\left[S\left(\frac{f}{F}\right)(y)\right]=(TF)(y)\chi\left(y,\frac{f}{F}(\phi(y))\right)=\frac{TF(y)}{F(\phi(y))}f(\phi(y)),\]
	as we wanted.
\end{proof}


\subsection{Jarosz' Theorem}\label{subsectionjarosz}

Throughout this subsection, we fix $\mathbb{K}=\mathbb{R}$ or $\mathbb{C}$. Given a locally compact Hausdorff space $X$, we let $C_c(X)=C_c(X,0)$, the vector space of compactly supported $\mathbb{K}$-valued functions.

\begin{theorem}[Jarosz \cite{MR1060366}]\label{theoremjarosz}
	If $T\colon C_c(X)\to C_c(Y)$ is a linear $\perp$-isomorphism, then there exist a homeomorphism $\phi\colon Y\to X$ and a continuous non-vanishing function $p\colon Y\to\mathbb{K}$ such that $Tf(y)=p(y)f(\phi(y))$ for all $f\in C_c(X)$ and $y\in Y$.
\end{theorem}
\begin{proof}
	\textbf{First assume that $X$ and $Y$ are compact}, and let us show that $f\neq 0$ everywhere if and only if $Tf\neq 0$ everywhere. Suppose otherwise, say $f(x)=0$, and we have two cases: first, if $f$ is constant on a neighbourhood of $x$, this means that $Z(f)\neq \varnothing$, and Theorem \ref{theoremdisjoint} implies that $Z(Tf)\neq \varnothing$, and in particular $Tf(y)=0$ for any $y\in Z(Tf)$.
	
	In the second case, if $f$ is not constant on any neighbourhood of $x$, an argument similar to the one in the proof of Theorem \ref{theoremmilgram} yields a contradiction to $Tf$ being bounded, so $Tf(y)=0$ for some $y\in Y$.
	
	The result follows in this case from the Li--Wong Theorem (Theorem \ref{theoremliwong}).
	
	\textbf{Now let $X$ and $Y$ be arbitrary locally compact Hausdorff.} Given $b\in C_c(X)$, set $T_b\colon C(\supp(b))\to C(\supp(Tb))$ as $T_bf=(Tf')|_{\supp(Tb)}$, where $f'$ is any element of $C_c(X)$ extending $f$. Note that $T_bf$ does not depend on the choice of $f'$, because $T$ is an additive $\perp$-isomorphism (and Theorem \ref{theoremdisjoint}).
	
	Since $f\perpp g$ if and only if $f|_{\supp(b)}\perpp g|_{\supp(b)}$ for all $b$, the previous case allows us to obtain functions $p^b$ and $\phi^b$ such that $Tf(y)=p^b(y)f(\phi^b(y))$ for all $y\in\supp(Tb)$. Clearly, if $b\subseteq b'$ then $p^{b'}|_{\supp(b)}=p^b$ and $\phi^{b'}|_{\supp(b)}=\phi^{b'}$. Thus defining $p$ and $\phi$ as $p(y)=p^b(y)$ and $\phi(y)=\phi^b(y)$ where $b\in C_c(X)$ is such that $y\in\supp(b)$ we obtain the desired maps.\qedhere
\end{proof}


\subsection{Banach--Stone Theorem}\label{subsectionbanachstone}

We use the same notation as in the previous subsection. Given a locally compact Hausdorff space $X$, endow $C_c(X)$ with the supremum norm: $\Vert f\Vert_\infty=\sup_{x\in X}|f(x)|$.

Recall that, by the Riesz--Markov--Kakutani Representation Theorem (\cite[Theorem 2.14]{MR584266}), continuous linear functionals on $C_c(X)$ correspond to (integration with respect to) regular Borel measures on $X$. As a consequence, the extremal points $T$ of the unit ball of the dual of $C_c(X)$ have the form $T(f)=\lambda f(x)$ for some $x\in X$ and $|\lambda|=1$.

Given $f\in C_c(X)$, denote by $N(f)$ the set of extremal points $T$ in the unit ball of the dual space $C_c(X)^*$ such that $T(f)\neq 0$. From the previous paragraph we obtain
\[f\perp g\iff N(f)\cap N(g)=\varnothing\tag{BS},\]
and the Banach--Stone Theorem is an immediate consequence of Jarosz' Theorem.

\begin{theorem}[Banach--Stone \cite{MR1501905}]\label{theorembanachstone}
	Let $X$ and $Y$ be locally compact Hausdorff spaces and let $T\colon C_c(X)\to C_c(Y)$ be an isometric linear isomorphism. Then there exists a homeomorphism $\phi\colon Y\to X$ and a continuous function $p\colon Y\to\mathbb{S}^1$ for which
	\[Tf(y)=p(y)f(\phi(y))\qquad\forall f\in C(X),\ \forall y\in Y.\]
\end{theorem}


\subsection{\texorpdfstring{$L^1$}{L¹}-spaces}\label{subsectionl1spaces}

Let $\mathbb{K}=\mathbb{R}$ or $\mathbb{C}$ be fixed. Given a topological space $X$, $C_c(X)$ will denote the space of $\mathbb{K}$-valued compactly supported continuous functions on $X$, where supports are the usual ones: $\supp(f)=\overline{[f\neq 0]}$.

Following \cite{MR924157}, a Borel measure $\mu$ on $X$ will be called \emph{regular} if
\begin{itemize}
	\item $\mu$ is locally finite (i.e., all compact sets have finite measure);
	\item For every Borel $E\subseteq X$, $\mu(E)=\inf\left\{\mu(V):E\subseteq V,\ V\text{ open}\right\}$;
	\item For every open $U\subseteq X$ with $\mu(U)<\infty$, $\mu(U)=\sup\left\{\mu(K):K\subseteq U,\ K\text{ compact}\right\}$.
\end{itemize}
and recall that the \emph{support} of $\mu$ is the set of points $x\in X$ whose neighbourhoods always have positive measure. We say that $\mu$ is \emph{fully supported} (on $X$) is the support of $\mu$ coincides with $X$, i.e., if every nonempty open subset has positive measure.

We will now prove that the $C_c(X)$ endowed with the $L^1$-norm of a fully supported measure $\mu$ completely determine both $X$ and $\mu$. The following lemma describes the $\perp$ relation in terms of the $L^1$-norm, and follows from elementary measure-theoretic considerations.

\begin{lemma}\label{lemmadisjointnessl1}
	Let $X$ be a locally compact Hausdorff space and $\mu$ a fully supported, locally finite Borel measure on $X$. If $\Vert\cdot\Vert_1$ denotes the corresponding $L^1$-norm, then for all $f,g\in C_c(X)$, $f\perp g$ if and only if
	\[\Vert Af+Bg\Vert_1=|A|\Vert f\Vert_1+|B|\Vert g\Vert_1\qquad\forall A,B\in\mathbb{K}\ntag\label{equationlemmadisjointnessl1}\]
\end{lemma}

\begin{theorem}\label{theoremdisjointnessl1}
	Let $X$ and $Y$ be locally compact Hausdorff spaces with fully supported regular Borel measures $\mu_X$ and $\mu_Y$, and let $T\colon C_c(X)\to C_c(Y)$ be a linear isomorphism which is isometric with respect to the $L^1$-norms. Then there exists a homeomorphism $\phi\colon Y\to X$ and a continuous function $p\colon Y\to\mathbb{S}^1$ such that
	\[Tf(y)=p(y)\frac{d\mu_X}{d(\phi_*\mu_Y)}(\phi(y))f(\phi(y))\]
	for all $f\in C_c(X)$ and $y\in Y$.
\end{theorem}
\begin{proof}
	By the previous lemma, $T$ is a $\perp$-isomorphism, so Jarosz' Theorem (\ref{theoremjarosz}) implies that there are a homeomorphism $\phi\colon Y\to X$ and a non-vanishing continuous function $P\colon Y\to\mathbb{C}$ such that $T(f)(y)=P(y)f(\phi(y))$ for all $f\in C_c(X)$ and $y\in Y$.
	
	Now using the fact that $T$ is isometric, we have, for every $f\in C_c(X)$,
	\[\int_X|f|d\mu_X=\int_Y|Tf|d\mu_Y=\int_Y|P||f\circ\phi|d\mu_Y=\int_X|P\circ\phi^{-1}||f|d(\phi_*\mu_Y)\]
	which means that $|P\circ\phi^{-1}|$ is a continuous instance of the Radon--Nikodym derivative $d\mu_X/d(\phi_*\mu_Y)$. Since $p=P/|P|\colon Y\to\mathbb{S}^1$ is continuous, we obtain the result.
\end{proof}

\subsection{Measured groupoid convolution algebras}\label{subsectionwendelstheorem}

In the next three results, we will focus on convolution algebras of topological groupoids. First, we will consider \emph{measured groupoids} in the sense of Hahn. See \cite{MR496796,MR496797,MR584266,MR1444088,MR0427598}. Note that throughout this section we consider only regular measures.

Recall that a \emph{groupoid} $\G$ is a small category with inverses, and a \emph{topological groupoid} is a groupoid endowed with a topology making the product and inversion maps continuous.

The \emph{source} and \emph{range} maps on $\G$ are defined as $\so(a)=a^{-1}a$ and $\ra(a)=aa^{-1}$, respectively. The \emph{unit space} of $\G$ is $\G[0]=\so(\G)$, and is identified with the object space of $\G$. We denote by $\G[2]=\left\{(a,b)\in\G\times\G:\so(a)=\ra(b)\right\}$ the set of \emph{composable pairs}, i.e., pairs $(a,b)$ for which the product $ab$ is defined. Given $x,y\in\G[0]$, we denote $\G^y=\ra^{-1}(x)$, $\G_x=\so^{-1}(x)$, and $\G_x^y=\G^y\cap\G_x$. We call $\G_x^x$ the \emph{isotropy group} at $x$. The product of two subsets $A,B\subseteq\G$ is $AB=\left\{ab:(a,b)\in(A\times B)\cap\G[2]\right\}$.

Common examples of topological groupoids are: Equivalence relations, topological groups (where $\G[0]$ is a singleton) and topological spaces (where $\G=\G[0]$). More generally, every continuous group action induces a \emph{transformation groupoid}.

Initially, given a locally compact Hausdorff topological groupoid $\G$, we consider $C_c(\G)$, the space of real or complex-valued, compactly supported, continuous functions on $\G$, simply as a vector space (with pointwise operations). Recall the notion of a \emph{Haar system}:

\begin{definition}[{\cite[Definition 2.2]{MR584266}}]\label{definitionhaarsystem}
	A (continuous) \emph{left Haar system} for a locally compact Hausdorff topological groupoid $\G$ is a collection of regular Borel measures $\lambda=\left\{\lambda^x:x\in\G[0]\right\}$ on $\G$ such that
	\begin{enumerate}[label=(\roman*)]
		\item\label{haarsystem1} For each $x\in\G[0]$, $\lambda^x$ has support contained in $\G^x$;
		\item\label{haarsystem2} (left invariance) For each $a\in\G$, $\lambda^{\ra(a)}(aE)=\lambda^{\so(a)}(E)$ for every compact $E\subseteq \G^{\so(a)}$;
		\item\label{haarsystem3} (continuity) For each $f\in C_c(\G)$, the map $\G[0]\to\mathbb{C}$, $x\mapsto \int fd\lambda^x$, is continuous.
	\end{enumerate}
	We will not make any distinction of whether each $\lambda^x$ is considered as a measure on $\G$ or as a measure on $\G^x$. We say that $\lambda$ is \emph{fully supported} if the support of $\lambda^x$ is all of $\G^x$ for all $x\in\G[0]$.
\end{definition}

Left invariance of $\lambda$ implies that for all $a\in\G$ and $f\in C_c(\G^{\ra(a)})$
\[\int f(s)d\lambda^{\ra(a)}(s)=\int f(at)d\lambda^{\so(a)}(t)\]
and we endow $C_c(\G)$ with convolution product
\[(fg)(a)=\int f(s)g(s^{-1}a)\lambda^{\ra(a)}(s)=\int f(at)g(t^{-1})\lambda^{\so(a)}(t),\]
which makes $C_c(\G)$ an algebra. It follows that for all $f,g\in C_c(\G)$,
\[\supp(fg)\subseteq\supp(f)\supp(g).\ntag\label{equationinclusionsupportproduct}\]

\begin{definition}
	Let $\G$ be a locally compact Hausdorff topological groupoid with a left Haar system $\lambda$. Given a regular Borel measure $\mu$ on $\G[0]$, the measure \emph{induced} by $\mu$ and $\lambda$ is the unique regular Borel measure $(\lambda\circ\mu)$ on $\G$ which satisfies
	\[(\lambda\circ\mu)(E)=\int_{\G[0]}\lambda^x(E)d\mu(x)\]
	for every compact $E\subseteq\G$. (The existence of $(\lambda\circ\mu)$ is guaranteed by the Riesz--Markov--Kakutani Representation Theorem.)
\end{definition}

If $\mu$ is fully supported on $\G[0]$ and $\lambda$ is a fully supported Haar system on a locally compact Hausdorff groupoid $\G$, then $(\lambda\circ\mu)$ is fully supported on $\G$.

The following lemma will allow us to verify if certain maps are groupoid morphisms.

\begin{lemma}\label{lemmaproduct}
	Given a topological groupoid $\G$ with $\G[0]$ Hausdorff and $a,b\in\G$, we have $\so(a)=\ra(b)$ if and only if for every pair of neighbourhoods $U$ of $a$ and $V$ of $b$ the product $UV$ is nonempty.
\end{lemma}
\begin{proof}
	From the second condition one can construct two nets $(a_i)_i$ and $(b_i)_i$ (over the same ordered set) converging to $a$ and $b$, respectively, such that $\so(a_i)=\ra(b_i)$, and so $\so(a)=\ra(b)$ because $\G[0]$ is Hausdorff. The reverse implication is trivial.
\end{proof}

\begin{lemma}\label{lemmaderivativehaar}
	If $\lambda$ and $\mu$ are continuous Haar systems on a locally compact Hausdorff topological groupoid $\G$ such that the Radon--Nikodym derivatives $D^x=\frac{d\lambda^x}{d\mu^x}$ exist for all $x\in\G[0]$, then $D$ is invariant in the sense that for all $a\in\G$ and $\mu^{\so(a)}$-almost every $g\in\G^{\so(a)}$, $D^{\ra(a)}(ag)=D^{\so(a)}(g)$.
\end{lemma}
\begin{proof}
	Using invariance of $\mu$ and $\lambda$, we have, for every $f\in C_c(\G^{\so(a)})$,
	\begin{align*}
	\int f(t) D^{\ra(a)}(at)d\mu^{\so(a)}(t)&=\int f(a^{-1}s)D^{\ra(a)}(s)d\mu^{\ra(a)}(s)\\
	&=\int f(a^{-1}s)d\lambda^{\ra(a)}(s)=\int f(t)d\lambda^{\so(a)}(t).
	\end{align*}
	Thus $t\mapsto D^{\ra(a)}(at)$ satisfies the property of the Radon--Nikodym derivative $d\lambda^{\so(a)}/d\mu^{\so(a)}=D^{\so(a)}$, hence these functions coincide $\mu^{\so(a)}$-a.e.\qedhere
\end{proof}

Now we prove that the convolution algebra $C_c(\G)$ together with the $L^1$-norm coming from $\lambda\circ\mu$, where $\lambda$ is a fully supported Haar system on $\G$ and $\mu$ is a fully supported measure on $\G[0]$ completely determines the triple $(\G,\lambda,\mu)$, up to isomorphism (compare to \cite{MR2102633}). We denote by $\mathbb{S}^1$ the \emph{circle group} (of complex numbers with absolute value $1$ under multiplication).

\begin{theorem}\label{theoremmeasuredgroupoidconvolutionalgebra}
	Let $\G$ and $\H$ be locally compact Hausdorff groupoids. For each $Z\in\left\{\G,\H\right\}$, let $\lambda_Z$ be a fully supported Haar system on $Z$, and $\mu_Z$ a fully supported regular Borel measure on $Z^{(0)}$.
	
	If $T\colon C_c(\G)\to C_c(\H)$ is an algebra isomorphism which is isometric with respect to the $L^1$-norms of $(\lambda_Z\circ \mu_Z)$ ($Z\in\left\{\G,\H\right\}$), then there are a topological groupoid isomorphism $\phi:\H\to\G$ and a continuous morphism $p:\H\to\mathbb{S}^1$ such that
	\[Tf(h)=p(h)D(\phi(h))f(\phi(h))\]
	where $D$ is a continuous instance of the Radon--Nikodym derivative
	\[D(a)=\frac{d\lambda_{\G}^{\ra(a)}}{d(\phi_*\lambda_{\H}^{\phi^{-1}(\ra(a))})}(a)\]
	and in this case, $\mu_{\G}=\phi_*\mu_{\H}$.
\end{theorem}
\begin{proof}
	Again applying Lemma \ref{lemmadisjointnessl1} and Jarosz' Theorem (\ref{theoremjarosz}), we can find a homeomorphism $\phi\colon\H\to\G$ and a continuous non-vanishing scalar function $P$ such that
	\[Tf(h)=P(h) f(\phi(h))\qquad\text{for all } f\in C_c(\G)\text{ and }h\in\H.\]
	
	Let us check that $\phi$ is a groupoid morphism. Suppose $(a,b)\in\H[2]$, and consider neighbourhoods $U$ and $V$ of $\phi(a)$ and $\phi(b)$, respectively.
	
	Choose non-negative functions $f_U,f_V\in C_c(\H)$ such that
	\[\supp(f_a)\subseteq\phi^{-1}(U),\quad\supp(f_b)\subseteq \phi^{-1}(V)\quad\text{and}\quad f_a(a)=f_b(b)=1.\]
	Then $ab\in\supp(f_af_b)$, because $\lambda_{\H}$ has full support, and so $\phi(ab)\in\supp(T^{-1}(f_af_b))$. As $\phi$ is the $T$-homeomorphism and $T$ is an isomorphism, the inclusion in \eqref{equationinclusionsupportproduct} implies $\phi(ab)\subseteq UV$. By Lemma \ref{lemmaproduct}, the product $\phi(a)\phi(b)$ is defined, and moreover, continuity of the product implies that every neighbourhood of $\phi(a)\phi(b)$ contains $\phi(ab)$. Since $\G$ is Hausdorff then $\phi(ab)=\phi(a)\phi(b)$. Therefore $\phi$ is a morphism and a homeomorphism, thus a topological groupoid isomorphism.
	
	We now need to rewrite the function $P$ as $P(h)=p(h)D(\phi(h))$ as in the statement of the theorem, and to this end we use multiplicativity of $T$ and compare integrals. If $f,g\in C_c(\G)$ and $c\in\H$, then on one hand
	\begin{align*}
	T(fg)(c)&=(TfTg)(c)=\int_{\G^{\ra(c)}} Tf(t)Tg(t^{-1}c)\lambda_{\H}^{\ra(c)}(t)\\
	&=\int_{\H^{\ra(c)}} P(t)f(\phi(t))P(t^{-1}c)g(\phi(t^{-1}c))d\lambda_{\H}^{\ra(c)}(t)\\
	&=\int_{\G^{\phi(\ra(c))}} P(\phi^{-1}(s))f(s)P(\phi^{-1}(s)^{-1}c)g(s^{-1}\phi(c))d(\phi_*\lambda_{\H}^{\ra(c)})(s)\ntag\label{theoremmeasuredgroupoidconvolutionalgebratodecomposeP1}
	\end{align*}
	and on the other
	\begin{align*}
	T(fg)(c)&=P(c)(fg)(\phi(c))=P(c)\int_{\G^{\phi(\ra(c))}} f(t)g(t^{-1}\phi(c))d\lambda_{\G}^{\phi(\ra(c))}(t)\ntag\label{theoremmeasuredgroupoidconvolutionalgebratodecomposeP2}
	\end{align*}
	
	Now let $f\in C_c(\G^{\phi(\ra(c))})$ be an arbitrary non-negative function. Define $g\in C_c(\G_{\phi(\so(c))})$ by $g(t)=f(\phi(c)t^{-1})$. Extending $f$ and $g$ arbitrarily to elements of $C_c(\G)$, Equations \eqref{theoremmeasuredgroupoidconvolutionalgebratodecomposeP1} and \eqref{theoremmeasuredgroupoidconvolutionalgebratodecomposeP2} become
	\[\int_{\G^{\phi(\ra(c))}} P(\phi^{-1}(s))P(\phi^{-1}(s)^{-1}c)|f(s)|^2d\phi_*\lambda_{\H}^{\ra(c)}(s)=P(c)\int_{\G^{\phi(\ra(c))}} |f(s)|^2d\lambda_{\G}^{\phi(\ra(c))}(s)\ntag\label{equationtheoremrenaultorsimilar}\]
	for all non-negative $f\in C_c(\G^{\phi(\ra(c))})$ and all $c\in\H$. Define $D\colon\G\to\mathbb{C}$ (or $\mathbb{R}$ in the real case) by
	\[D(s)=\frac{P(\phi^{-1}(s))P(\phi^{-1}(s)^{-1})}{P(\phi^{-1}(\ra(s)))}.\]
	Using Equation \eqref{equationtheoremrenaultorsimilar} with $y=\ra(c)$ in place of $c$, we obtain
	\[\int_{\G^{\phi(y)}} D(s)|f(s)|^2d(\phi_*\lambda_{\H}^{y})(s)=\int_{\G^{\phi(y)}}|f(s)|^2d\lambda_{\G}^{\phi(y)}(s)\]
	for all $f\in C_c(\G^{\phi(y)})$, thus $D$ is a continuous instance of the Radon--Nikodym derivative
	\[D(s)=\frac{d\lambda_{\G}^{\phi(y)}}{d(\phi_*\lambda_{\H}^{y})}(s)=\frac{d\lambda_{\G}^{\phi(\ra(c))}}{d(\phi_*\lambda_{\H}^{\ra(c)})}(s).\]
	Now applying this to Equation \eqref{equationtheoremrenaultorsimilar}, and using regularity of all measures involved, we conclude that for $\lambda_{\G}^{\phi(\ra(c))}$-a.e.\ $s\in\G^{\phi(\ra(c))}$
	\[P(\phi^{-1}(s))P(\phi^{-1}(s)^{-1}c)=D(s)P(c)\]
	and since all functions involved are continuous, and $\lambda_{\G}^{\phi(\ra(c))}$ has full support, the same equality is actually valid for \emph{all} $s\in\G^{\phi(\ra(c))}$. Equivalently, for all $c\in\H$ and all $t\in\H^{\ra(c)}$, $P(t)P(t^{-1}c)=D(\phi(t))P(c)$. Together with Lemma \ref{lemmaderivativehaar}, this implies that the map $p=P/(D\circ\phi)$ is a continuous groupoid morphism from $\H$ to the group of non-zero scalars.
	
	Now let us verify that $\mu_{\G}$ and $\phi_*\mu_{\H}$ are equivalent measures. By regularity of the measures, we may extend the formula $T(f)=P\cdot (f\circ\phi)$ to obtain a linear isomorphism between the classes of measurable functions on $\G$ and on $\H$, and which restricts to an isometry from $L^1(\lambda_{\G}\circ\mu_{\G})$ to $L^1(\lambda_{\H}\circ\mu_{\H})$. Suppose that $K\subseteq\G[0]$ has positive $\mu_{\G}$-measure, and let $f=1_{\ra^{-1}(K)}$ be the characteristic function of $\ra^{-1}(K)$. Then $Tf=P\cdot 1_{\phi^{-1}(\ra(K))}$. As $\Vert f\Vert_{\G}>0$, then $\Vert Tf\Vert_{\H}>0$ as well, so the support of $Tf$ has positive $(\lambda_{\H}\circ\mu_{\H})$-measure, and this implies that $\ra(\supp(Tf))=\phi^{-1}(K)$ has positive $\mu_{\H}$-measure. Therefore $\mu_{\G}$ is absolutely continuous with respect to $\phi_*\mu_{\H}$. The reverse implication is similar.
	
	To prove that $p$ takes value in $\mathbb{S}^1$, let us denote by $\Vert\cdot\Vert_Z$ the $L^1$-norm with respect to $(\lambda_Z\circ\mu_Z)$ when $Z\in\left\{\G,\H\right\}$. For all $f\in C_c(\G)$ we have
	\begin{align*}
	\Vert Tf\Vert_{\H}&=\int_{\H[0]}\left(\int_{\H^y}|Tf|d\lambda_{\H}^y\right)d\mu_{\H}(y)=\int_{\H[0]}\left(\int_{\H^y}D|p(f\circ\phi)|d\lambda_{\H}^y\right)d\mu_{\H}(y)\\
	&=\int_{\G[0]}\left(\int_{\G^{x}}|(p\circ\phi^{-1})f|d\lambda_{\G}^x\right)d(\phi_*\mu_{\H})(x)\\
	&=\int_{\G[0]}\left(\int_{\G^{x}}|(p\circ\phi^{-1})(s)|\left(\frac{d(\phi_*\mu_{\H})}{d\mu_{\G}}(x)\right)|f(s)|d\lambda_{\G}^x(s)\right)d\mu_{\G}(x)\\
	&=\int_{\G[0]}\left(\int_{\G^{x}}|(p\circ\phi^{-1})(s)|\left(\frac{d(\phi_*\mu_{\H})}{d\mu_{\G}}(\ra(s))\right)|f(s)|d\lambda_{\G}^x(s)\right)d\mu_{\G}(x)\\
	&=\int_{\G}|p\circ\phi^{-1}|\left(\frac{d(\phi_*\mu_{\H})}{d\mu_{\G}}\circ\ra\right)|f(s)|d(\lambda_{\G}\circ\mu_{\G})
	\end{align*}
	and since $\Vert Tf\Vert_{\H}=\Vert f\Vert_{\G}=\int_{\G}|f|d(\lambda_{\G}\circ\mu_{\G})$, we obtain
	\[|p\circ\phi^{-1}|=\frac{d\mu_{\G}}{d(\phi_*\mu_{\H})}\circ\ra\qquad(\lambda_{\G}\circ\mu_{\G})\text{-a.e.}\ntag\label{equationcomparecocycleandradonnykodymofbasisspace}\] 
	Since $p$ is a morphism then $p(\H^{(0)})=\left\{1\right\}$, which, along with Equation \eqref{equationcomparecocycleandradonnykodymofbasisspace} and continuity of $p$, yields $|p|=|p\circ\ra|=1$ on $\H$. The same  Equation (\ref{equationcomparecocycleandradonnykodymofbasisspace}) then also implies $\mu_{\G}=\phi_*\mu_{\H}$.
\end{proof}

\begin{named}{Remark}
	In the case of groups, the same type of classification was first proven by Wendel in \cite{MR0049910}, when considering the whole $L^1$-algebras of locally compact Hausdorff groups instead of only algebras of compactly supported continuous functions. Further generalizations of Wendel's Theorem were proven in \cite{MR0177058} and \cite{MR0193531}, and closely results in \cite{MR0160846} and \cite{MR0361622}.
\end{named}

\subsection{\texorpdfstring{$(I,\ra)$}{(I,r)}-Groupoid convolution algebras}

In the next result we will again use the convolution algebras of topological groupoids, however now we will consider another norm, which was already defined in the work of Hahn (\cite{MR496797}) and played an important role in Renault's work (\cite{MR584266}). A locally compact Hausdorff groupoid is \emph{étale} if the range map $\ra\colon\G\to\G[0]$ is a local homeomorphism. From this, it follows that $\G[0]$ is open in $\G$, that the product map is open and that $\G^x$ is discrete for all $x\in\G[0]$ (see \cite{MR2304314}).

Let $\G$ be a locally compact étale Hausdorff groupoid, $\mathbb{K}=\mathbb{R}$ or $\mathbb{C}$ and $\theta=0$, and let $\lambda$ be a Haar system for $\G$. Again, we will consider the convolution algebra $C_c(\G)=C_c(\G,\mathbb{K})$ as defined in the previous subsection.

Every left Haar system on an étale groupoid is essentially the counting measure (\cite[2.7]{MR584266}), in the sense that for all $x,y\in\G[0]$, the map $a\mapsto \lambda^{\ra(a)}(\left\{a\right\})$ is constant on set $\G_x^y$.

We define the $(I,\ra)$-norm on $C_c(\G)$ as
\[\Vert f\Vert_{I,\ra}=\sup_{x\in\G[0]}\int |f|d\lambda^x.\]

As $\G$ is Hausdorff, the unit space $\G[0]$ of $\G$ is a closed subgroupoid of $\G$, hence (trivially) étale, Hausdorff and locally compact itself. The convolution product on $C_c(\G[0])$ coincides with the pointwise product, and the $(I,\ra)$-norm is the uniform one: $\Vert f\Vert_{I,\ra}=\Vert f\Vert_\infty=\sup_{x\in\G[0]}|f(x)|$.

Moreover, $\G[0]$ is also open in $\G$ (because $\G$ is étale), so we can identify $C_c(\G[0])$ with the subalgebra $\left\{f\in C_c(\G[0]:\supp(f)\subseteq\G[0]\right\}$ of $C_c(\G)$.

\begin{definition}
	The algebra $C_c(\G[0])$, identified as a subalgebra of $C_c(\G)$, is called the \emph{diagonal subalgebra} of $C_c(\G)$. If $\G$ and $\H$ are locally compact étale Hausdorff groupoids, an isomorphism $T\colon C_c(\G)\to C_c(\H)$ is called \emph{diagonal-preserving} if $T(C_c(\G[0]))=C_c(\H[0])$.
\end{definition}

\begin{theorem}\label{theoremrenault}
	Let $\G$ and $\H$ be locally compact Hausdorff étale groupoids with continuous fully supported left Haar systems $\lambda_G$ and $\lambda_H$, respectively, and $T\colon C_c(\G)\to C_c(\H)$ a diagonal-preserving algebra isomorphism, isometric with respect to the $(I,\ra)$-norms. Then there is a (unique) topological groupoid isomorphism $\phi\colon\H\to\G$ and a continuous morphism $p\colon\H\to\mathbb{S}^1$ such that
	\[Tf(h)=p(h)D(\phi(h))f(\phi(h))\]
	where $D$ is a continuous instance of the Radon--Nikodym derivative
	\[D(a)=\frac{d\lambda_{\G}^{\ra(a)}}{d(\phi_*\lambda_{\H}^{\phi^{-1}(\ra(a))})}(a).\]
\end{theorem}
\begin{proof}
	By the Banach--Stone Theorem (\ref{theorembanachstone}), there is a homeomorphism $\phi\colon\H[0]\to\G[0]$ and a continuous function $P\colon\H[0]\to\mathbb{S}^1$ such that $Tf(y)=P(y)f(\phi(y))$ for all $f\in C_c(\G[0])$ and $y\in\H[0]$. Since $T$ is multiplicative we obtain $P=1$. (The same conclusion can be obtained in a similar manner by Milgram's or Jarosz' Theorem.)
	
	For each $x\in\G[0]$, let $\left\{a^x_i:i\in I_x\right\}$ be a net of functions in $C_c(\G[0])$ satisfying:
	\begin{enumerate}[label=(\roman*)]
		\item\label{theoremrenaultauxiliary(i)} $0\leq a^x_i\leq 1$, and $a^x_i(x)=1$;
		\item\label{theoremrenaultauxiliary(ii)} $\bigcap_i\supp(a^x_i)=\left\{x\right\}$;
		\item\label{theoremrenaultauxiliary(iii)} If $j\geq i$ then $[a^x_j\neq 0]\subseteq[a^x_i\neq 0]$.
	\end{enumerate}
	Items \ref{theoremrenaultauxiliary(ii)}-\ref{theoremrenaultauxiliary(iii)} and compactness of each $\supp(a^x_i)$ imply that $\left\{[a^x_i\neq 0]:i\in I_x\right\}$ is a neighbourhood basis at $x$. For $y\in\H[0]$, let $a^y_i=T(a^{\phi(y)}_i)=a^{\phi(y)}_i\circ\phi$, so that the net $\left\{a^y_i:i\in I_{\phi(y)}\right\}$ satisfies \ref{theoremrenaultauxiliary(i)}-\ref{theoremrenaultauxiliary(iii)} as well.
	
	Continuity of $\lambda_G$ implies that for all $x\in \G[0]$ and $f\in C_c(\G)$, $\lim_{i\in I_x}\Vert a^x_i f\Vert_{I,\ra}=\int |f(\gamma)|d\lambda^x(\gamma)$, and similarly on $\H$.
	
	Given $f,g\in C_c(\G)$, we use Lemma \ref{lemmadisjointnessl1} to obtain
	\begin{align*}
	f\perp g&\iff\forall x\left(f|_{\G^x}\perp g|_{\G^x}\text{ in }C(\G^x)\right)\\
	&\iff\forall x\forall A,B(\lim\Vert a^x_i(Af+Bg)\Vert_{I,\ra}=\lim(|A|\Vert a^x_i f\Vert_{I,\ra}+|B|\Vert a^x_i g\Vert_{I,\ra}
	\end{align*}
	and the last condition is preserved by $T$, so by Jarosz' Theorem $T$ is of the form $Tf(\alpha)=\widetilde{P}(\alpha)f(\widetilde{\phi}(\alpha))$ for a certain homeomorphism $\widetilde{\phi}\colon\H\to\G$ and a non-vanishing continuous scalar function $\widetilde{P}$. We can readily see that $\widetilde{\phi}$ and $\widetilde{P}$ are extensions of $\phi$ and $P$, respectively, so instead let us simply denote $\widetilde{\phi}=\phi$ and $\widetilde{P}=P$.
	
	The proof that $\phi$ is a groupoid isomorphism, and that $P$ can be decomposed as $P=(D\circ\phi)p$ for the (continuous) Radon--Nikodym derivative $D$ and some continuous morphism $p\colon\H\to\mathbb{C}\setminus\left\{0\right\}$ is the same as in Theorem \ref{theoremmeasuredgroupoidconvolutionalgebra}, but the verification that $|p|=1$ is different.
	
	Given $y\in\H[0]$  and $f\in C_c(\G)$, using the definition of $D$ as a Radon--Nikodym derivative,
	\[\int_{\H^y}|Tf|d\lambda_{\H}^y=\int_{\H^y} |p|(D\circ\phi)|f\circ\phi|d\lambda_{\H}^y=\int_{\G^{\phi^{-1}(y)}} |p\circ\phi^{-1}||f|d\lambda_{\G}^{\phi(y)}.\]
	Considering again the functions $a^{\phi(y)}_i$ and $a^y_i$, and the fact that $T$ is isometric we obtain
	\begin{align*}
	\int_{\G^{\phi^{-1}(y)}} |p\circ\phi^{-1}||f|d\lambda_{\G}^{\phi(y)}&=\lim_{i\in I_{\phi(y)}}\Vert a_i^y Tf\Vert_{I,\ra}=\lim_{i\in I_{\phi(y)}}\Vert a_i^{\phi(y)} f\Vert_{I,\ra}= \int |f|d\lambda_{\G}^{\phi(y)}
	\end{align*}
	for all $f\in C_c(\G)$, which implies that $|p|=1$ $\lambda_{\H}^y$-a.e. Since $p$ is continuous and $\lambda_{\H}$ is fully supported, we conclude that $|p|=1$ on $\H$.
\end{proof}

\subsection{Steinberg Algebras}

Steinberg algebras were independently introduced in \cite{MR2565546} and \cite{MR3274831}, as algebraic analogues of groupoid C*-algebras, and are generalizations of Leavitt path algebras and universal inverse semigroup algebras. We refer to \cite{MR3743184} and \cite{carlsenrout2017} for more details.

A locally compact, zero-dimensional étale groupoid is called \emph{ample}. A \emph{bisection} of a groupoid $\G$ is a subset $A\subseteq\G$ such that the source and range maps are injective on $A$. If $\G$ is an ample Hausdorff groupoid, we denote by $\KB(\G)$ the semigroup of compact-open bisections of $\G$, which forms a basis for the topology of $\G$.

In this section, $R$ is a fixed commutative ring with unit. Given an ample Hausdorff groupoid $\G$, we denote by $R^{\G}$ the $R$-module of $R$-valued functions on $\G$. Given $A\subseteq\G$, we define $1_A$ as the characteristic function of $A$ (with values in $R$).

Steinberg algebras were The goal of this section is to prove that the Steinberg algebra of an ample Hausdorff groupoid $\G$ together with its diagonal algebra completely characterize $\G$. Although the main theorem of this subsection (Theorem \ref{theoremsteinbergalgebras}) is partially stated and proven (for more general \emph{graded} Steinberg algebras) in \cite[Corollary 3.14]{carlsenrout2017}, we can obtain a precise classification of the diagonal-preserving isomorphisms of Steinberg algebras, as described in Theorem \ref{theoremsteinbergalgebras} and Corollary \ref{corollaryclassificationofautomorphismsofsteinbergalgebras}

We will need to recover the bisections of $\G$ from $A_R(\G)$, and in particular the compact-open subsets of $\G[0]$. The main idea is, again, to identify subsets of $\G[0]$ with their characteristic functions, and these are precisely the functions which attain only the values $0$ and $1$. We thus need to assume an extra condition on the ring $R$.

\begin{definition}[{\cite[X.7]{MR1878556}}]
	A (nontrivial) commutative unital ring $R$ is \emph{indecomposable} if its only idempotents are $0$ and $1$. Equivalently, $R$ is indecomposable if it cannot be written as a direct sum $R\simeq R_1\oplus R_2$, where $R_1$ and $R_2$ are nontrivial rings.
\end{definition}

A subset $A$ of a groupoid $\G$ is a bisection if and only if $AA^{-1}\cup A^{-1}A\subseteq \G[0]$. A similar type of condition will be used to recover an ample Hausdorff groupoid $\G$ from the pair $(A_R(\G),D_R(\G))$.

\begin{definition}
	A \emph{normalizer} of $D_R(\G)$ is an element $f\in A_R(\G)$ for which there exists $g\in A_R(\G)$ such that
	\begin{enumerate}
		\item[(i)] $fD_R(\G)g\subseteq D_R(\G)$ and $gD_R(\G)f\subseteq D_R(\G)$;
		\item[(ii)] $fgf=f$ and $gfg=g$.
	\end{enumerate}
	We denote by $N_R(\G)$ the set of normalizers of $D_R(\G)$. An element $g$ satisfying (i) and (ii) above will be called and \emph{inverse of $f$ relative to $D_R(\G)$}.
\end{definition}

It can be verified that $N_R(\G)$ is a multiplicative subsemigroup of $A_R(\G)$, which is moreover an \emph{inverse} semigroup. In particular, the inverse relative to $D_R(\G)$ of an element $f\in N_R(\G)$ is unique. However, we will not necessitate these results.

\begin{example}
	If $A\in\KB(\G)$ then $1_A\in N_R(\G)$. More generally, if $\lambda_1,\ldots,\lambda_n$ are invertible elements in $R$ and $U_1,\ldots,U_n$ are compatible disjoint compact-open bisections (that is, $\bigcup_i U_i$ is also a bisection), then $f=\sum_i \lambda_i1_{U_i}$ is a normalizer of $D_R(\G)$. The unique inverse of $f$ relative to $D_R(\G)$ is given by $f^*=\sum_i\lambda_i^{-1}1_{U_i^{-1}}$, that is, $f^*(a)=f(a^{-1})^{-1}$ for all $a\in\supp(f)^{-1}$.
\end{example}

In order to recover $\G$ from $(A_R(\G),D_R(\G))$, we need that all normalizers of $D_R(\G)$ have the form described in the example above, so additional conditions will have to be assumed on the groupoids we consider.

The following property was considered in \cite{arxiv1711.01903v2}, when working on the same recovery problem.

\begin{definition}\label{definitionlocalbisectionhypothesis}
	If $\G$ is an ample Hausdorff groupoid and $R$ is an indecomposable (commutative, unital) ring, we say that $(\G,R)$ satisfies the \emph{local bisection hypothesis} if $\supp(f)$ is a bisection for all $f\in N_R(\G)$.
\end{definition}

\begin{lemma}\label{lemmaformofnrgforlocalbisectionhypothesis}
	Suppose that $(\G,R)$ satisfies the local bisection hypothesis and $f\in N_r(\G)$. Then for all $a\in\supp(f)$, $f(a)$ is invertible in $R$.
\end{lemma}
\begin{proof}
	Let $g$ be an inverse of $f$ relatively to $D_R(\G)$. First note that $fg=f1_{\so(\supp(f))}g\in D_R(\G)$.
	
	Let $a\in\supp(f)$. Since $fg$ is an idempotent in $D_R(\G)$,  the product in $D_R(\G)$ is pointwise and $R$ is indecomposable, then $fg(\ra(a))\in\left\{0,1\right\}$. Moreover, as $\supp(f)$ is a bisection we have $f(a)=fgf(a)=(fg)(\ra(a))f(a)$, so $(fg)(\ra(a))=1$.
	
	Again using that $\supp(f)$ is bisection, we obtain
	\[1=fg(\ra(a))=f(a)g(a^{-1})\]
	so $f(a)$ is invertible in $R$.\qedhere
\end{proof}

The following stronger condition was considered in \cite{carlsenrout2017}, and is more easily checked than the one above. Recall that if $R$ is a ring and $G$ is a group, a \emph{trivial unit} of the group-ring $RG$ is an element of the form $ug$, where $u$ is invertible in $R$ and $g\in G$.

\begin{definition}
	If $\G$ is an ample Hausdorff groupoid and $R$ is an indecomposable (commutative, unital) ring, we say that $(\G,R)$ \emph{satisfies condition (S)} if the set of all $x\in\G[0]$ such that the group ring $R\G_x^x$ has only trivial units is dense in $\G[0]$.
\end{definition}

The property of a group-ring $RG$ (where $G$ is a group and $R$ is a ring) having only trivial units has been studied, for example, in \cite{MR0002137}. A group $G$ is \emph{indexed} if there exists a non-trivial group morphism from $G$ to $\mathbb{Z}$, and \emph{indicable throughout} if every nontrivial finitely generated subgroup of $G$ is indexed. (Note that if $G$ is indicable throughout then $G$ is torsion-free.)

\begin{theorem}[{\cite[Theorem 13]{MR0002137}}]
	If $G$ is indicable throughout and $R$ is an integral domain, then $RG$ has only trivial units.
\end{theorem}

Every free group, and every torsion-free abelian group is indicable throughout. The class of indicable throughout groups is closed under products, free products and extensions (see \cite{MR0002137}).

The following result from \cite{carlsenrout2017} provides a large class of groupoids satisfying the local bisection hypothesis. Although in \cite{carlsenrout2017} the authors assume stronger hypotheses (namely, that $R$ is an integral domain and $R\G_x^x$ does not have zero divisors for all $x$ in a dense subset of $\G[0]$), their proof works under the weaker assumptions we adopt.

\begin{lemma}[{\cite[Lemma 3.5(2)]{carlsenrout2017}}]\label{lemmanormalizerssteinbergalgebradiagonal2}
	Suppose that $\G$ is an ample Hausdorff groupoid, $R$ is an indecomposable ring and that $(\G,R)$ satisfies condition (S). Then $(\G,R)$ satisfies the local bisection hypothesis.
\end{lemma}

An important class of groupoids consists of the \emph{topologically principal} ones, whose associated algebras have been extensively studied (see, for example, \cite{MR3189105,MR2745642,MR1681679,MR2590626}). In fact it is possible to classify C*-algebras which come from them (see \cite{MR2460017}).

\begin{definition}\label{definitiontopologicallyprincipal}
	A topological groupoid $\G$ is \emph{topologically principal} if the set of all $x\in X$ whose isotropy group $\G_x^x$ is trivial is dense in $\G[0]$.
\end{definition}

It follows that if $\G$ is an ample Hausdorff topologically principal groupoid and $R$ is an indecomposable ring, then $(\G,R)$ satisfies the local bisection hypothesis.

We are ready to classify diagonal-preserving isomorphisms of Steinberg algebras of groupoids and rings satisfying the local bisection hypothesis. For this, let us first define the class of maps of interest:

\begin{definition}\label{definitioncocycle}
	Let $R$ and $S$ be rings and $\G$ be a groupoid. Denote by $\operatorname{Iso}_+(R,S)$ the set of additive isomorphisms from $R$ to $S$. A map $\chi\colon\G\to\operatorname{Iso}_+(R,S)$ satisfying $\chi(ab)(rs)=\chi(a)(r)\chi(b)(s)$ for all $(a,b)\in\G[2]$ and $r,s\in R$ will be called a \emph{cocycle}.
\end{definition}

\begin{example}
	Consider $C_2=\left\{1,g\right\}$, the group of order 2, acting on itself by left multiplication and consider the transformation groupoid $\G=C_2\ltimes C_2$. Let $R=S=\mathbb{Z}$. If we define $\chi\colon\G\to\operatorname{Iso}_+(R,S)$ by $\chi(1,y)(r)=r$ and $\chi(g,r)=-r$, then $\chi$ is a cocycle. Note that $\chi(g,1)$ is not a ring isomorphism.
\end{example}

\begin{example}
	Suppose $R$ is a unital ring and $\chi\colon\G\to\operatorname{Iso}_+(R,R)$ is a cocycle. Then $\chi$ is a morphism from the groupoid $\G$ to the group (under composition) $\operatorname{Iso}_+(R,R)$ if, and only if, $\chi(x)=\id_R$ for all $x\in \G[0]$.
\end{example}

\begin{proposition}
	Let $R$ and $S$ be commutative unital rings, $\G$ a groupoid and $\chi\colon\G\to\operatorname{Iso}_+(R,S)$ a cocycle. Then
	\begin{enumerate}
		\item[(a)] For all $x\in\G[0]$, $\chi(x)$ is a ring isomorphism;
		\item[(b)] For all $a\in\G$, if $u\in R$ is invertible then $\chi(a)(u)$ is invertible in $S$, and $\chi(a)(u)^{-1}=\chi(a^{-1})(u)$.
		\item[(c)] For all $a\in\G$, $\chi(\so(a))=\chi(\ra(a))$. In other words, the restriction of $\chi$ to $\G[0]$ is \emph{invariant}.
	\end{enumerate}
\end{proposition}
\begin{proof}
	The cocycle condition states that $\chi(ab)(rs)=\chi(a)(r)\chi(b)(s)$ for all $a,b,r,s$. Taking $a=b=x$ yields (a). Taking $b=a^{-1}$, $r=u$ and $s=u^{-1}$ yields (b), and for item (c) we use commutativity of $S$:
	\[\chi(\so(a))(r)=\chi(a^{-1}a)(1r)=\chi(a^{-1})(1)\chi(a)(r)=\chi(a)(r)\chi(a^{-1})(1)=\chi(\ra(a))(r).\qedhere\]
\end{proof}

We endow $\operatorname{Iso}_+(R,S)$ with the topology of pointwise convergence, so that a map $\chi$ from a topological space $X$ to $\operatorname{Iso}_+(R,S)$ is continuous if and only if for every $r\in R$, the map $X\ni x\mapsto\chi(x)(r)\in S$ is continuous, that is, locally constant.

\begin{theorem}\label{theoremsteinbergalgebras}
	Let $\G$ and $\H$ be ample Hausdorff groupoids. Let $R$ and $S$ be two indecomposable (commutative, unital) rings such that $(\G,R)$ and $(\H,S)$ satisfy the local bisection hypothesis. Let $T\colon A_R(\G)\to A_S(\H)$ be a diagonal-preserving ring isomorphism, that is, $T(D_R(\G))=D_S(\H)$.
	
	Then there exists a unique topological groupoid isomorphism $\phi\colon\H\to\G$ and a continuous cocycle $\chi\colon\H\to\operatorname{Iso}_+(R,S)$ such that $Tf(a)=\chi(a)(f(\phi(a)))$ for all $a\in\H$ and $f\in A_R(\G)$.
\end{theorem}
\begin{proof}
	Since $T$ preserves the respective diagonal algebras, it also preserves their normalizers, i.e., $T(N_R(\G))=N_S(\H)$. Let us describe disjointness first for elements in $N_R(\G)$. The local bisection hypothesis implies, by Lemma \ref{lemmaformofnrgforlocalbisectionhypothesis}, that an element $f$ of $N_R(\G)$ has the form
	\[f=\sum_{i=1}^n\lambda_i1_{U_i}\]
	where $\lambda_1,\ldots,\lambda_n$ are invertible elements in $R$ and $U_1,\ldots,U_n$ are disjoint compact-open bisections of $\G$ such that $\bigcup_{i=1}^n U_i=\supp(f)$ is also a compact-open bisection. A similar statement holds for $N_S(\H)$.
	
	If $f,g\in N_R(\G)$, then $f\subseteq g$ if and only if $f=g p$ for some $p\in D_R(\G)$: Indeed,
	\begin{itemize}
		\item If $f=gp$ then $\supp(f)\subseteq\supp(g)\supp(p)\subseteq\supp(g)$;
		\item Conversely, if $\supp(f)\subseteq\supp(g)$ take $p=g^*f$. Then
		\[\supp(p)\subseteq\supp(g^*)\supp(f)\subseteq(\supp(g))^{-1}\supp(g)\subseteq\G[0]\]
		where the last inclusion follows from $\supp(g)$ being a bisection. The equality $f=gp$ follows from the definition of $p$ and since $\ra(\supp(f))\subseteq\ra(\supp(g))$.
	\end{itemize}
	
	Therefore $T$ preserves inclusion of normalizers. Since $N_R(\G)$ contains $\left\{1_U:U\in\KB(\G)\right\}$ then it is regular (Definition \ref{definitionregularperpp}), because $\KB(\G)$ is a basis for the topology of $\G$. Hence $T$ also preserver disjointness of normalizers (Proposition \ref{propositionrelationsrelations}).
	
	To prove that $T$ preserves disjointness in all of $A_R(\G)$, we decompose elements of $A_R(\G)$ in terms of elements of $N_R(\G)$ and $D_R(\G)$: if $f,g\in A_R(\G)$, then $f\perp g$ if and only if there are finite collections of normalizers $f_i,g_j\in N_R(\G)$ and elements $\widetilde{f_i},\widetilde{g_j}\in D_R(\G)$ ($1\leq i\leq n$, $1\leq j\leq m$) such that
	\[f=\sum_i f_i\widetilde{f_i},\quad g=\sum_j g_j\widetilde{g_j}\text{ and }f_i\perp g_j\text{ for all }i,j\]
	Indeed, if there are such $f_i,g_j,\widetilde{f_i},\widetilde{g_j}$ then $\supp(f)\subseteq\bigcup_i\supp(f_i)$ and $\supp(g)\subseteq\bigcup_j\supp(g_j)$, and the latter sets are disjoint.
	
	Conversely, we write $f=\sum_i\lambda_i1_{A_i}$, where the $A_i$ are pairwise disjoint compact-open bisections and $\lambda_i\neq 0$, and take $f_i=1_{A_i}$ and $\widetilde{f_i}=\lambda_i1_{\so(A_i)}$, so that $\supp(f)=\bigcup_{i=1}^n\supp(f_i)$. Similarly, writing $g=\sum_j\widetilde{g_j}g_j$ where $g_j\in N_R(\G)$ and $\supp(g)=\bigcup_j\supp(g_j)$, then $f\perp g$ implies $f_i\perp g_j$ for all $i$ and $j$.
	
	Therefore, $T$ is a $\perp$-isomorphism. Note that $\perpp$ and $\perp$ coincide in $A_R(\G)$, since its elements are locally constant (similarly to Example \ref{examplekaniarmoutil}). Then $T$ is a $\perpp$-isomorphism, so let $\phi\colon\H\to\G$ be the $T$-homeomorphism. The verification that $\phi$ is a groupoid isomorphism is similar to that of Theorem \ref{theoremmeasuredgroupoidconvolutionalgebra}.
	
	Since elements of $A_R(\G)$ (and $A_S(\H)$) are locally constant, then for all $f\in A_R(\G)$,
	\[f(\phi(a))=0\iff \phi(a)\in Z(f)\iff x\in Z(Tf)\iff Tf(a)=0.\]
	and therefore $T$ is basic (by additivity of $T$ and Proposition \ref{propositionbasicnessofgroupvalued}). Let $\chi$ be the $T$-transform. Since $T$ is additive with the pointwise operations, each section $\chi(\alpha)=\chi(\alpha,\cdot)$ is additive (by Proposition \ref{propositionmodelmorphism}). This yields a map $\chi\colon\G\to\operatorname{Iso}_+(R,S)$, and we need now to verify that $\chi$ is a cocycle.
	
	If $(a,b)\in\H[2]$ and $r,s\in R$, choose compact-open bisections $U,V$ of $\G$ containing $\phi(a)$ and $\phi(b)$, respectively. Then using multiplicativity of $T$ we obtain
	\begin{align*}
	\chi(ab)(rs)&=\chi(ab)\big((r1_U)(s1_V)(\phi(ab))\big)=T\big((r1_U)(s1_V)\big)(ab)\\
	&=\big(T(r1_U)T(s1_V)\big)(ab)=\sum_{cd=ab}T(r1_U)(c)T(s1_V)(d)\\
	&=\sum_{cd=ab}\chi(c)\big(r1_U(\phi(c))\big)\chi(d)\big(s1_V(\phi(d))\big)
	\end{align*}
	If $cd=ab$ is such that the last term above is nonzero, then $\ra(c)=\ra(a)$ and $\phi(c)\in U$, so since $U$ is a bisection we obtain $a=c$. Similarly, $d=b$, therefore $\chi(ab)(rs)=\chi(a)(r)\chi(b)(s)$, and $\chi$ is a cocycle.
	
	It remains only to prove that $\chi$ is continuous: Let $r\in R$ be fixed, $a\in\H$ and $U$ any compact-open bisection containing $\phi(a)$. For all $b\in\phi^{-1}(U)$,
	\[\chi(b)(r)=\chi(b)(r1_U(\phi(b)))=T(r1_U)(b)\]
	which means that the map $b\mapsto\chi(b)(r)$ coincides with $T(r1_U)$ on $\phi^{-1}(U)$ and thus it is continuous.
\end{proof}

We should note that according to \cite{arxiv1711.01903v2}, the local bisection hypothesis is preserved by diagonal-preserving isomorphisms, so the same result is valid if we assume, in principle, that only $(\G,R)$ satisfies this condition.

From this we can immediately classify the group of diagonal-preserving automorphisms of Steinberg algebras satisfying the local bisection hypothesis. Let $\G$ be a groupoid and $R$ a ring. Denote by $\operatorname{Coc}(\G,R)$ the set of all continuous cocycles $\chi\colon\G\to\operatorname{Iso}_+(R,R)$, which is a group with the canonical (pointwise) structure: $(\chi\rho)(a)=\chi(a)\circ\rho(a)$ for all $\chi,\rho\in C(\G,R)$ and $a\in\G$, where $\circ$ denotes composition.

Let $\operatorname{Aut}(\G)$ be the group of topological groupoid automorphisms of $\G$. Then $\operatorname{Aut}(\G)$ acts on $\operatorname{Coc}(\G,R)$ in the usual (dual) manner: for $\phi\in\operatorname{Aut}(\G)$, $\chi\in\operatorname{Coc}(\G,R)$ and $a\in\G$ set $(\phi\chi)(a)=\chi(\phi^{-1}a)$.

Denote by $\operatorname{Aut}(A_R(\G),D_R(\G))$ the group of diagonal-preserving ring automorphisms of $A_R(\G)$. From Theorem \ref{theoremsteinbergalgebras} we immediately obtain:

\begin{corollary}\label{corollaryclassificationofautomorphismsofsteinbergalgebras}
	If $(\G,R)$ satisfies the local bisection hypothesis, then the group $\operatorname{Aut}(A_R(\G),D_R(\G))$ is isomorphic to the semidirect product $C(\G,R)\rtimes\operatorname{Aut}(\G)$.
\end{corollary}

\subsection{Groups of circle-valued functions}

A natural question in C*-algebra theory is whether we can extend isomorphisms of unitary groups of C*-algebras to isomorphisms (or anti/conjugate-isomorphisms) of the whole C*-algebras. Dye proved in \cite{MR0066568} that this is always possible for continuous von Neumann factors, however this is not true in the general C*-algebraic case, even in the commutative case. Therefore we should consider isomorphisms between unitary groups which preserve more structure than just the product, such as an analogue to that of Theorem \ref{theoremliwong}. Recall that the unitary group of a commutative C*-algebra $C(X)$, where $X$ is compact Hausdorff, is $C(X,\mathbb{S}^1)$.

\begin{theorem}\label{theoremisomorphismgroupifcirclevaluedfunctions}
	Let $X$ and $Y$ be two Stone (zero-dimensional, compact Hausdorff) spaces. Suppose that $T\colon C(X,\mathbb{S}^1)\to C(Y,\mathbb{S}^1)$ is a group isomorphism such that $1\in f(X)\iff 1\in Tf(X)$. Then there exist a homeomorphism $\phi\colon Y\to X$, a finite isolated subset $F\subseteq Y$ and a continuous function $p\colon Y\setminus F\to\{\pm 1\}$ satisfying $Tf(y)=f(\phi(y))^{p(y)}$ for all $y\in Y\setminus F$.
	
	In particular, if $X$ (and/or $Y$) do not have isolated points then $F=\varnothing$.
\end{theorem}

The following lemma, based on \cite{MR3162258}, will be crucial to the proof of the theorem.

\begin{lemma}\label{lemmatheoremisomorphismgroupifcirclevaluedfunctions}
	Suppose that $X$ is a Stone space. For every pair of continuous functions $f,g\colon X\to\mathbb{S}^1$ and for every finite subset $F\subseteq X$ such that $f(F)\cup g(F)$ does not contain $1$, there exists $h\in C(X,\mathbb{S}^1)$ such that
	\[h(x)\not\in\left\{f(x),g(x)\right\}\text{ for all }x\text{ and }h(F)=\{1\}.\]
\end{lemma}
\begin{proof}
	For every point $y\in F$, choose a clopen set $U_y$ containing $y$ such that $f(U_y)\cup g(U_y)$ does not contain $1$. For every other point $x\in X':=X\setminus\bigcup_{y\in F}U_y$, there is a clopen set $U\subseteq X'$ such that $f(U)\cup g(U)\neq\mathbb{S}^1$. Using compactness of $X'$ and taking complements and intersections if necessary we can find a clopen partition $U_1,\ldots,U_n$ of $X'$ such that $f(U_i)\cup g(U_i)\neq\mathbb{S}^1$ for all $i$. Simply choose $z_i\in \mathbb{S}^1\setminus(f(U_i)\cup g(U_i))$ and define $h=z_i$ on $U_i$, and $h=1$ on $\bigcup_{f\in F}U_f$.
\end{proof}

\begin{proof}[Proof of Theorem \ref{theoremisomorphismgroupifcirclevaluedfunctions}]
	For the notion of support we will use (Definition \ref{definitionsigma}), we take $\theta=1$, the constant function at $1$, so regularity of $C(X,\mathbb{S}^1)$ is immediate. 
	
	Suppose that $f\perp g$ but that $Tf$ and $Tg$ are not disjoint. By Lemma \ref{lemmatheoremisomorphismgroupifcirclevaluedfunctions}, there exists $H\in C(Y,\mathbb{S}^1)$ such that $H\neq Tf,Tg$ everywhere, but that $1\in H(Y)$. Let $h=T^{-1}H$. Then $T(f^{-1}h)$ and $T(g^{-1}h)$ do not attain $1$, which implies that $f^{-1}h$ and $g^{-1}h$ do not attain $1$ as well. Thus
	\begin{align*}
	h^{-1}(1)&=X\cap h^{-1}(1)=(f^{-1}(1)\cup g^{-1}(1))\cap h^{-1}(1)\\
	&=(g^{-1}(1)\cap h^{-1}(1))\cup(f^{-1}(1)\cap h^{-1}(1))\subseteq (g^{-1}h)^{-1}(1)\cup (f^{-1}h)^{-1}(1)=\varnothing.
	\end{align*}
	But $(Th)^{-1}(1)=H^{-1}(1)$ is nonempty, contradicting the given property of $T$.
	
	Therefore $f\perp g$ implies $Tf\perp Tg$, and the same argument yields the opposite implication, so $T$ is a $\perp$-isomorphism. Let $\mathcal{A}(X)$ and $\mathcal{A}(Y)$ be the subgroups of order-$2$ elements of $C(X,\mathbb{S}^1)$ and $C(Y,\mathbb{S}^1)$, respectively (i.e., the groups of continuous functions with values in $\left\{-1,1\right\}$).
	
	$\mathcal{A}(X)$ and $\mathcal{A}(Y)$ are also regular, since $X$ and $Y$ are zero-dimensional, and the restriction $T|_{\mathcal{A}(X)}\colon\mathcal{A}(X)\to \mathcal{A}(Y)$ is a $\perpp$-isomorphism, because $\perp$ and $\perpp$ coincide on $\mathcal{A}(X)$ and $\mathcal{A}(Y)$. Let $\phi\colon Y\to X$ be the corresponding $T|_{\mathcal{A}(X)}$-homeomorphism. 
	
	Let $h\in C(X,\mathbb{S}^1)$ be arbitrary. Since $\sigma(h)=\bigcup_{a\in \mathcal{A}(X),a\subseteq h}\sigma(a)$ and $T$ is a $\perp$-isomorphism, we obtain by Theorem \ref{theoremdisjoint} that, for all $h\in C(X,\mathbb{S}^1)$,
	\[\phi(\sigma(Th))=\bigcup_{\substack{a\in \mathcal{A}(X)\\a\subseteq h}}\phi(\sigma(Ta))=\bigcup_{\substack{a\in \mathcal{A}(X)\\a\subseteq h}}\sigma(a)=\sigma(h)\]
	Since $\phi$ is a homeomorphism it preserves closures, from which it follows that $T$ is also a $\perpp$-isomorphism, and $\phi$ is also the $T$-homeomorphism.
	
	\begin{description}
		\item[Claim] $f(\phi(y))=1\iff Tf(y)=1$.
	\end{description}
	
	Suppose $f(\phi(y))\neq 1$. Choose a function $g\in C(X,\mathbb{S}^1)$ which coincides with $f$ on a neighbourhood of $\phi(y)$ and such that $1\not\in g(X)$. Then $1\not\in Tg(Y)$ and since $Tf$ coincides with $Tg$ on a neighbourhood of $y$ then $Tf(\phi(y))=Tg(\phi(y))\neq 1$. The other direction is analogous, and thus we have proved the claim.
	
	Therefore $T$ is basic. Let $\chi$ be the $T$-transform, so that each section $\chi(y,\cdot)$ is an automorphism of the circle. If $\chi(y,\cdot)$ is continuous then it has the form $\chi(y,z)=z^{p(y)}$ where $p(y)\in\left\{\pm1\right\}$. Let us prove that for all except finitely many $y\in Y$, the section $\chi(y,\cdot)$ is continuous.
	
	Let $F=\left\{y\in Y:\chi(y,\cdot)\text{ is discontinuous}\right\}$, and suppose that $F$ were infinite. By Proposition \ref{propositiondisjointopensets}, there are countably infinitely many distinct points $y_n\in F$ ($n\in\mathbb{N}$), such that no $y_n$ lies in the closure of the other ones. We can choose a sequence $z_n\to 1$ such that $\chi(y_n,z_n)$ lies in the second quadrant of the circle. Define $f(\phi(y_n))=z_n$, $f=1$ on the boundary of $\{\phi(y_n):n\in\mathbb{N}\}$ and extend $f$ continuously to all of $X$. Let $y$ be a cluster point of $\{y_n\}_n$, so that in particular $f(\phi(y))=1$. Then
	\[Tf(y)=\chi(y,1),\quad Tf(y_n)=\chi(y_n,z_n)\]
	But $y$ is an accumulation point of the $y_n$, and $Tf(y_n)$ lies in the second quadrant while $Tf(y)=1$, a contradiction to the continuity of $Tf$.
	
	Therefore $F$ is finite, so now we show that it is open in order to conclude that its points are isolated in $Y$. Let $y\in Y$ and choose $z_0\in\mathbb{S}^1$ of the form $z_0=e^{it}$ where $-\pi/4\leq t\leq \pi/4$, but such that $\chi(y,z_0)$ is in the second or third quadrant, so in particular it is not $z_0$ nor $z_0^{-1}$. Denote by $z_0$ the constant function at $z_0$, we that
	\[T(z_0)(y)=\chi(y,z_0)\neq z_0,z_0^{-1}.\]
	Since $T(z_0)$ is continuous, there is a neighbourhood $U$ of $y$ such that $\chi(x,z_0)\neq z_0,z_0^{-1}$ for all $x\in U$, so $x\in F$.
	
	Therefore $Y'=Y\setminus F$ is also compact, and we already constructed the function $p\colon Y'\to\{\pm 1\}$ with the desired property. To see that $p$ is continuous, denote by $i$ the constant function $x\mapsto i$ and note that
	\[p^{-1}(1)=\left\{y\in Y':\chi(y,i)=i\right\}=\left\{y\in Y':T(i)(y)=i\right\}=T(i)^{-1}(i)\cap Y'\]
	and similarly $p^{-1}(-1)=T(i)^{-1}(-i)\cap Y'$, so these two sets, which are complementary in $Y'$, are closed and hence clopen.
\end{proof}

\begin{example}
	As an easy example where the subset $F\subseteq Y$ in the previous theorem is nonempty, let $X=Y=\left\{\ast\right\}$ be (equal) singletons, and let $t\colon\mathbb{S}^1\to\mathbb{S}^1$ be a discontinuous automorphism of $\mathbb{S}^1$.
	
	Consider the map $T\colon C(X,\mathbb{S}^1)\to C(Y,\mathbb{S}^1)$, $T(f)(\ast)=t(f(\ast))$ (in other words, $T$ is the function obtained from $t$ by identifying $C(X,\mathbb{S}^1)$ and $C(Y,\mathbb{S}^1)$ with $\mathbb{S}^1$). Then $T$ satisfies the hypotheses of the previous theorem but $F=Y$.
\end{example}

We now endow $C(X,\mathbb{S}^1)$ with the uniform metric:
\[d_\infty(f,g)=\sup_{x\in X}|f(x)-g(x)|\]
(which is the metric coming from the C*-algebra $C(X,\mathbb{C})$).

\begin{theorem}\label{theoremisometricisomorphismcirclegroups}
	If $X$ and $Y$ are as above and $T\colon C(X,\mathbb{S}^1)\to C(Y,\mathbb{S}^1)$ is an isometric isomorphism, then there is a homeomorphism $\phi\colon Y\to X$ and a continuous function $p\colon Y\to\left\{\pm 1\right\}$ such that $Tf(y)=f(\phi(y))^{p(y)}$ for all $y\in Y$.
\end{theorem}
\begin{proof}
	We identify each $\lambda\in\mathbb{S}^1$ with the corresponding constant map on $X$ or $Y$. The constant function $-1$ is characterized by the following two properties:
	\begin{itemize}
		\item $(-1)^2=1$;
		\item If $g^3=1$, then $d_\infty(-1,g)\in\left\{1,2\right\}$.
	\end{itemize}
	Thus $T(-1)=-1$. A function $f$ does not attain $1$ if and only if $d_\infty(-1,f)<1$, so $T$ preserves functions not attaining $1$, and we apply Theorem \ref{theoremisomorphismgroupifcirclevaluedfunctions} (or more precisely its proof) in order to obtain a homeomorphism $\phi\colon Y\to X$, a function $\chi\colon Y\times\mathbb{S}^1\to\mathbb{S}^1$ and a continuous function $p\colon Y'\to\left\{-1,1\right\}$, where $Y'=\left\{y\in Y:\chi(y,\cdot)\text{ is continuous}\right\}$, such that
	\[Tf(y)=\chi(y,f(\phi(y))\qquad\text{and}\qquad\chi(y',t)=t^{p(y')}\]
	for all $y\in Y$, $y'\in Y'$ and $f\in C(X,\mathbb{S}^1)$. It remains only to prove that $Y'=Y$, i.e., every section $\chi(y,\cdot)$ is continuous.
	
	If $\lambda_i\to\lambda$ in $\mathbb{S}^1$ then we also have uniform convergence of the corresponding constant functions, so
	\[\chi(y,\lambda_i)=T(\lambda_i)y\to T(\lambda)y=\chi(y,\lambda)\]
	thus $\chi(y,\cdot)$ is continuous for all $y$.
\end{proof}

\subsection*{Acknowledgements}
This work is part of the author's PhD thesis at the University of Ottawa, written under supervision of Thierry Giordano and Vladimir Pestov, who provided many useful insights and suggestions to the subject at hand.

\bibliographystyle{abbrv}
\bibliography{library}
\end{document}